\documentclass[10pt]{article}     % -*- latex -*-
\usepackage{amsthm}
\usepackage{amsfonts, amssymb, amsmath, amscd, latexsym}
\usepackage{epsfig}
\usepackage[latin1]{inputenc}
\newtheorem{thm}{Theorem}[section]
\newtheorem{lemma}[thm]{Lemma}
\newtheorem{cor}[thm]{Corollary}

\newtheorem{remark}[thm]{Remark}

\newtheorem{definition}[thm]{Definition}

\numberwithin{equation}{section}

\usepackage{color}

\def\RR{\mathbb R}
\def\ds{\displaystyle}
\def\R{\boldsymbol{R}}
\def\Q{\boldsymbol{Q}}
\def\GP{\boldsymbol{GA}}

\def\Gupiu{\boldsymbol{Gu}}
\def\Gumeno{\boldsymbol{gu}}
\def\Gspiu{\boldsymbol{Gs}}
\def\Gsmeno{\boldsymbol{gs}}
\def\Hupiu{\boldsymbol{Gu}}
\def\Humeno{\boldsymbol{gu}}
\def\Hspiu{\boldsymbol{Gs}}
\def\Hsmeno{\boldsymbol{gs}}
\def\Ls{\boldsymbol{L2}}
\def\Lu{\boldsymbol{L1}}

\def\x{\boldsymbol{x}}
\def\la{\lambda}

\def\ep{\varepsilon}
\newcommand{\eu}{\textrm{e}}
\newcommand{\lir}{{\lim_{r \to \infty}}}
\newcommand{\liro}{{\lim_{r \to 0}}}
\newcommand{\lit}{{\lim_{t \to +\infty}}}
\newcommand{\lito}{{\lim_{t \to -\infty}}}
\newcommand\bs[1]{{\boldsymbol{#1}}}

\DeclareMathAccent{\cll}{\mathalpha}{operators}{'027}

\DeclareMathOperator{\Intera}{Int}

\newcommand{\Rsol}{$\mathcal {R}$}
\newcommand{\Ssol}{$\mathcal {S}$}
\newcommand{\fdsol}{${\sf fd}$}
\newcommand{\sdsol}{${\sf sd}$}

\newcommand\sol[3]{${\cal #1}\!\stackrel{{#2}}{\_}\!{\sf #3d}$}

\begin{document}
\title{Entire solutions of superlinear problems with indefinite weights and Hardy potentials.}
\author {Matteo Franca\thanks{Dipartimento di Scienze Matematiche,
Universit\`a Politecnica delle Marche, Via Brecce Bianche 1, 60131 Ancona -
Italy. Partially supported by G.N.A.M.P.A. - INdAM (Italy) and
MURST (Italy)} \hspace{1mm} and
Andrea Sfecci
\thanks{Dipartimento di Scienze Matematiche,
Universit\`a Politecnica delle Marche, Via Brecce Bianche 1, 60131 Ancona -
Italy. Partially supported by G.N.A.M.P.A. - INdAM (Italy)}
}
\date{}
  \maketitle
\pagestyle{myheadings}

\begin{abstract}
We provide the structure of regular/singular fast/slow decay radially symmetric solutions for a class of superlinear elliptic equations with an indefinite weight on the nonlinearity $f(u,r)$. In particular we are interested in the case where $f$ is positive
in a ball and negative outside, or in the reversed situation.
 We extend the approach to elliptic equations in presence of Hardy potentials. By the use of Fowler transformation we study the corresponding dynamical systems, presenting the construction of invariant manifolds when the global existence of solutions is not ensured.
\end{abstract}
\vspace{5mm} \textbf{Key Words:} $\,$ supercritical equations, Hardy potentials, radial
solutions, regular/singular ground
states, Fowler transformation, invariant manifold, continuability.\\
\textbf{MR (2010) Subject Classification}: $\,$ 35j75, 35j91, 37d10, 34c37.
   \vspace{5mm}
% \markboth{\textsc{author}}{\textsc{title}}
\section{Introduction}
A first   purpose of this paper is to study the properties of  radial solutions for equations of the form
\begin{equation} \label{laplace}\tag{L}
\Delta u+f(u,r)=0
\end{equation}
where $u: \RR^n \to \RR$, $n>2$,  $r=|x|$,   and $f:\RR \times (0,+\infty)\to \RR$ is a differentiable function  which is null for $u=0$
  and super-linear in $u$.
Since we just deal with radial solutions we will indeed consider the following singular ordinary differential equation
\begin{equation} \label{eq.na}\tag{Lr}
u''+ \frac{n-1}{r} u' +f(u,r)=0\,,
%(u'(r)r^{n-1})'+f(u,r) r^{n-1}=0 \, ,
\end{equation}
where, abusing the notation, we have set $u(r)=u(x)$ for $|x|=r$, and
$'$ denotes differentiation with respect to $r$.

 Radial solutions play a key role for~\eqref{laplace}, since in many cases, e.g. $k(r)\equiv K>0$,
 positive solutions have to be radial (but also in many situations in which $k$ is allowed to vary, see e.g.~\cite{B1,GNN,LN}).
 They are also crucial to determine the threshold between fading and blowing up initial data in the associated parabolic problem,
 see e.g.~\cite{DLL,WWW}.

 In this article we   are mainly interested in classifying positive and nodal solutions
when $f(u,r)$ is  negative  for $r$ small and  positive for $r$ large,
or in the opposite situation.
The prototype for the nonlinearity we are interested in takes the form:
\begin{equation} \label{regular}
f(u,r)=k(r) u|u|^{q-2}
\end{equation}
where $k$ is a continuous   function,  which is either  negative  in a ball and  positive  outside and $q>2^*$,
or we have the reversed sign condition and $2_*<q<2^*=\frac{2n}{n-2}$, where  $2_*:=\frac{2(n-1)}{n-2}<2^*:=\frac{2n}{n-2}$,
are respectively  the Serrin and the Sobolev critical exponents.

The behavior of solutions of~\eqref{eq.na}, with $f$ as in~\eqref{regular},
changes drastically according to the sign of $k$, when $q>2$.
When $k(r)<0$ for any $r>0$,  positive solutions are convex,
and their maximal interval of continuation may
be bounded either from above or from below or both.
On the other hand if $k(r)>0$ all the solutions of~\eqref{eq.na} are continuable for any $r>0$;
further, if $k(r)>0$ the structure of positive solutions of~(\ref{eq.na}) changes drastically when the exponent $q$
 in~(\ref{regular}) passes through
 some critical values, such as    $2_*$  and  $2^*$,
  see e.g. \cite{F2,NS}.
 In fact new and more complex situations arise  when the non-linearity   exhibits
 both subcritical an supercritical behavior with respect to these exponents, see e.g.
 \cite{BDF,BJ2,B1, Fjdde,Fdie,Y1996}, for a far from being exhaustive bibliography.
 In fact we have an interaction between the exponent
 and the asymptotic behavior of $k$.
  Roughly speaking, if $f$ is of type~(\ref{regular}) and $k(r)$ behaves like a positive power,
 then the critical exponents get smaller, while they get larger  if $k$ behaves like a negative power.
 E.g., if $k(r)= r^{\delta}$ then the Sobolev critical exponent becomes $2^*_{\delta}=2\frac{n+\delta}{n-2}$.

With very weak assumptions, definitively  positive solutions exhibit two behaviors as $r \to 0$
 and as $r \to \infty$ when $k(r)>0$. Namely  $u(r)$ may be
 a \emph{regular solution}, i.e. $u(0)=d>0$ and $u'(0)=0$,
or a \emph{singular solution}, i.e. $\liro u(r)=+\infty$;
a \emph{fast decay (f.d.) solution}, i.e. $\lir u(r) r^{n-2}=L$,
 or a \emph{slow decay (s.d.) solution}, i.e. $\lir u(r)r^{n-2}=+\infty$.
 We emphasize that in many situations the behavior of singular and slow decay
 solutions   can be specified better  (cf. Remark~\ref{sing}).

 In the whole paper we use the following notation: we denote by $u(r,d)$ the regular solution of~(\ref{eq.na})
such that $u(0,d)=d$, and by $v(r,L)$ the fast decay solution such that
$\lir r^{n-2}v(r,L)=L$.
Moreover we call \emph{ground states} (G.S.),  regular positive solutions $u(r)$ defined for any $r \ge 0$ and such that $\lir u(r)=0$, and
\emph{singular ground states} (S.G.S.) positive singular solutions $v(r)$
defined for any $r > 0$ and such that $\liro v(r)=\infty$, $\lir v(r)=0$.

\medbreak

In this paper we continue the discussion, begun with~\cite{FmSa},
which is mainly focused on the case where $f$ is of type~\eqref{regular}, $q>2^*$, $k(r)$ is discontinuous and equals
$1$ inside a ball and $-1$ outside.  This kind of equation is a special reaction-diffusion equation, where
the reaction, modeled by $f$, is assumed to have a source effect inside a ball and an absorption effect outside.
So it can describe, e.g., the temperature $u$ in presence of a nonlinear reaction producing energy (taking place in a bounded box) and its inverse absorbing it (taking place in the environment where the box is immersed), both  heat regulated.
As specified in~\cite{FmSa} it can also describe the density of a substance
subject to diffusion and to a nonlinear reaction and its inverse, see also  \cite[§7]{Mu2}. The inhomogeneity
may be induced by the presence of an activator or an inhibitor.

In~\cite{FmSa} the purpose was to prove existence and exact multiplicity for regular solutions with f.d. and with s.d., and
to deduce their nodal properties, but just for a very specific  example, discontinuous in $r$.

Here we want to show that the case described in~\cite{FmSa} is the prototype for a large class of nonlinearities $f$. So we
 relax the requirement on $k(r)$;  in particular we   assume it  to be smooth, and we extend the results
to a wider   family of potentials $f$, whose
 main representative is given by
 \begin{equation}\label{unaeffe}
 \begin{array}{lll}
 f(u,r)= K(r)     u|u|^{q-2}  & q>2^* &(a) \\
  f(u,r)= -K(r)    u|u|^{q-2}  & 2_*<q<2^* &(b)
 \end{array}
 \end{equation}
 where  $K(r)$ changes sign one time: in particular $K(r)<0$ if $r<R$, $K(r)>0$ if $r>R$, for a certain $R$.
 Section~\ref{MAP} will be devoted to a deeper analysis of the possible nonlinearities we can  deal with, but we wish to emphasize that our result is new for this kind of nonlinearity, too.

 \medbreak

We emphasize that the presence of G.S. with f.d. is due to the coexistence of source and absorption effects
 (i.e. $f$ changes sign).
In literature there
 are many results on the structure of radial solutions for Laplace equations
with indefinite weights $k$,  see e.g.~\cite{Bae2,BJ2,CFG}.  However, these papers
are concerned with phenomena which are found when $k$ is a positive function,
and which persist even if
$k$ becomes      negative in some regions.
The structure results we find can just take place if we have a change in the
sign of $k$: if $q$ is either smaller or larger than $2^*$ there are no G.S. with fast decay, neither if $k(r) \equiv K>0$, nor if
  $k(r) \equiv K<0$. In fact, the structure of the solutions of~\eqref{eq.na} described in Corollaries \ref{HC} and \ref{HD}, reminds
  of the situation in which $q=2^*$ and we have a positive $k$ which behaves like a positive power for $r$ small and
  a negative power for $r$ large, see e.g.~\cite{DF,Y1996}.
  In the same  direction goes \cite{BGM}, which
  proves existence results (using a variational approach) which hold  just when
  the
  nonlinearities have  sign-changing weights; however in  \cite{BGM} the authors consider
  bounded domains and just in the subcritical case.  Further  in \cite{DMM,MaMa,Serena1} and in references therein
   the reader can find several nice and sharp structure results for sign-changing nonlinearities, even for more general operators   ($p$-Laplace, relativistic and mean curvature),
   in the framework of oscillation (and non-oscillation) theory,
  but for exterior domains, i.e. for solutions defined, say for $r>1$.

Performing this generalization with respect to~\cite{FmSa}   we pay two prizes: firstly
we can just give existence and multiplicity results, but we lose the control of uniqueness and exact multiplicity
since we mainly ask for asymptotic conditions; secondly  we have to face many technical problems and the discussion is more involved.
The major one is the following: asking for $f$ to be negative and possibly superlinear for $u$ large, we allow the existence
of non-continuable solutions, whose presence causes almost no difficulties for the special nonlinearities considered in~\cite{FmSa},
but it is a crucial problem here.

In fact the presence of non-continuable solutions, which are typical for the nonlinearity considered, rises a challenging
problem from the theoretical point of view. In fact
 we cannot apply the already established
invariant manifold theory  for non-autonomous systems
(see e.g. \cite{CoLe,Co78,Jsell}).
  In the appendix we perform a first step in order to   extend this theory
to the case where  non-continuable solutions are allowed.
 As far as we are aware
 this is the first time where such a problem is considered, and we think this can be a contribution
 from a methodological point of view to invariant manifold theory for non-autonomous dynamical
 systems.

\medbreak

Our analysis is directly performed for the following more general differential equation
\begin{equation}\label{Hlaplace}\tag{H}
\Delta u + \frac{h(r)}{r^2} u +f(u,r) = 0\,,
\end{equation}
and for its radial counterpart
\begin{equation}\label{hardy}\tag{Hr}
u''+ \frac{n-1}{r} u'+ \frac{h(r) }{r^2}u +f(u,r)=0\,.
\end{equation}
We assume that $h$ is a differentiable function satisfying the following requirement, which will be assumed
in the whole paper:
  $$\textbf{H} \qquad \qquad \begin{array}{rl}
  h(r) < \frac{(n-2)^2}{4} \,, & \text{for every } r\in (0,+\infty)\,, \\
  h(0)=\eta< \frac{(n-2)^2}{4} \, , &  \frac{h(r)-\eta}{r} \in L^1(0,1] \,, \\
 \lir h(r)= \beta < \frac{(n-2)^2}{4}\,,&  \frac{h(r)-\beta}{r} \in L^1[1,+\infty) \, .
   \end{array} $$
   The introduction in Laplace equation of the additional term $\frac{h(r)}{r^2} u$, often referred to as \emph{Hardy potential}, has raised
   a great interest recently, see e.g.~\cite{Bae1,FGazz,Terracini}, and we think is another main point of interest in our paper.
    Usually in literature the case $h \equiv \eta$
   is considered, with the requirement that $\eta \le \eta_c := \frac{(n-2)^2}{4}$ (the value $\eta_c$ is again critical).
   The restriction $\eta \le\eta_c$ is necessary in order to have definitively positive solutions, see e.g.~\cite{Cirstea}, either for $r$
small or for $r$ large, and $\eta_c$ can be interpreted as the first eigenvalue of the $\Delta u+u/r^2$, see Section~1 in~\cite{Terracini}.  Here we give a dynamical interpretation
of this assumption.
 Equation \eqref{hardy} has been subject to deep investigation for different type of $f$, see e.g. \cite{Bae1,FMT,FMT1,FGazz,Terracini}.
Usually   $h$ is assumed to be a constant, and there are very few results concerning the case where $h$ actually varies;
 however Terracini in~\cite{Terracini} and Felli et al. in~\cite{FMT} considered the case where $h$ is a function, depending in fact on its
angular coordinates (to model a magnetic field).

A  consequence of the presence of the Hardy term is a shift on the Serrin critical exponent, and the appearance of
a new critical value in the supercritical regime. More precisely, if $h(r) \equiv \eta< \eta_c=\frac{(n-2)^2}{4}$ we define
\begin{equation}\label{def2eta}
2_*(\eta):=2 \frac{n+\sqrt{(n-2)^2-4 \eta}}{n-2+\sqrt{(n-2)^2-4\eta}}
\end{equation}
(which gives back $2_*$ if $\eta=0$), and
\begin{equation}\label{defIeta}
{\rm I}(\eta):= \begin{cases}
+\infty & \text{if } \eta\leq 0\\
2 \frac{n-\sqrt{(n-2)^2-4 \eta}}{n-2-\sqrt{(n-2)^2-4\eta}} & \text{if } 0<\eta<\frac{(n-2)^2}{4}\,,
\end{cases}
\end{equation}
see, e.g.~\cite{Cirstea}.
Notice that $ \lim_{\eta\to\eta_c} 2_*(\eta)= 2^* = \lim_{\eta\to\eta_c} {\rm I}(\eta)$.

The presence of the Hardy term, affects greatly the asymptotic behavior of the solution.
In fact if $f(u,r)>0$ we continue to have two possible behavior for definitively positive solutions
either for $r$ small or for $r$ large.
Let us   set
\begin{equation}\label{defKappa}
\kappa(\eta):= \frac{(n-2)- \sqrt{(n-2)^2-4\eta}}{2} \,;
\end{equation}
we  introduce the following terminology.

 \begin{definition}\label{defRFS}
\begin{itemize}
\item A \Rsol-solution $u(r,d)$ satisfies $\liro u(r,d) r^{\kappa(\eta)} = d\in\RR$, while a \Ssol-solution $u$ satisfies $\liro u(r) r^{\kappa(\eta)} = \pm\infty$.
\item A \fdsol-solution $v(r,L)$ satisfies $\lir v(r,L)r^{n-2-\kappa(\eta)} = L\in\RR$, while a \sdsol-solution $u$ satisfies $\lir u(r)r^{n-2-\kappa(\eta)} = \pm \infty$.
\item a \sol Rkf solution $u(r,d)=v(r,L)$ is both a \Rsol-solution and a \fdsol-solution having $k$ nondegenerate zeros. We define similarly  \sol Rks solution   $u(r,d)$, \sol Skf solution $v(r,L)$. When we do not indicate the value $k$, e.g.  \sol S{}f,
we mean any solution with these asymptotic properties disregarding its number of zeroes.

\end{itemize}
\end{definition}

In case of equation~\eqref{eq.na} we can recognize respectively regular fast decay, regular slow decay and singular fast decay solutions having $k$ nondegenerate zeros.

Note that $0<\kappa(\eta) < \frac{n-2}{2}$ if $\eta>0$ and that $\kappa(\eta)<0$ if $\eta <0$, therefore
\Rsol-solutions  are  unbounded if $\eta>0$. We also emphasize that
bounded solutions do not exist for $\eta>0$, and that \Ssol-solutions are anyway larger than \Rsol-solutions for $r$ small,
see Remarks~\ref{differenze1},~\ref{differenze1.5},~\ref{differenze3} for more details.

This fact may cause relevant problems in applying
variational or functional techniques, but in fact it  finds an easy explanation with our approach.
However the structure of positive solutions is not greatly altered by the presence of the   Hardy potential
so that we can give here a  unified approach for both~\eqref{laplace} and~\eqref{Hlaplace}.

\medbreak

Consider a function $f$ as in \eqref{unaeffe}; we state the following assumption for $K$:
\begin{description}
  \item[\bf{K}] Assume that $K$ is $C^1$ and there is $R>0$ such that
   $K(r)<0$ for $0<r<R$ and $K(r)>0$ for $r>R$.
  Further assume that
  \begin{equation}\label{Kasympt}
  \begin{array}{ccccc}
    K(r)&=&K(0) r^{\delta_0} +o(r^{\delta_0}) & \textrm{ as } r \to 0 \, , & \textrm{ and } \\
     K(r)&=&K(\infty) r^{\delta_{\infty}} +o(r^{\delta_{\infty}}) & \textrm{ as } r \to \infty \, . & \,
     \end{array}
  \end{equation}
  where $K(0)<0<K(\infty)$ and $\delta_0,\delta_{\infty}>-2$. Further  there is $\varpi>0$ (small) such that
  $\liro r^{-\varpi}\frac{d}{dr} [K(r) r^{-\delta_0}]=0$, and $\lir  r^{\varpi}\frac{d}{dr} [K(r) r^{-\delta_{\infty}}]=0$.

\end{description}
Note that the weak assumption on the derivative of $K$ is just technical.
We need to introduce the following parameters which take into account of the shift
on the critical exponent due to the presence of the spatial dependence:
\begin{equation}\label{elleUeS}
    l=l(q,\delta)=2 \frac{q+\delta}{2+\delta}
\end{equation}
and notice that $l(q,0)=q$.

We postpone the statement of our main results in a more general framework to Section~\ref{secST}. We just propose here two corollaries which follow directly from Theorems~\ref{tm3} and~\ref{tm4} and apply to nonlinearities introduced in~\eqref{unaeffe}.

\begin{cor}\label{HC}
Assume $\bs{\rm H}$, let $f$ be as in~(\ref{unaeffe}a) and suppose $K(r)$ satisfies \textbf{K}.
Set $l_u=l(q,\delta_0)$ and $l_s=l(q,\delta_{\infty})$, and assume $2_*(\eta)<l_u<{\rm I}(\eta)$,
$2^*<l_s<{\rm I}(\beta)$.
Then there is an increasing sequence  $(A_k)_{k\geq 0}$  such that
 $u(r,A_{k})$ is a \sol Rkf.
Moreover
 $u(r,d)$ is a \sol R0s for any $0<d<A_0$, and there is $A^*_k \in [A_{k-1},A_k)$ such that
 $u(r,d)$ is a \sol Rks whenever $A^*_k<d<A_{k}$, for any $k \ge 1$.
\end{cor}
\begin{cor}\label{HD}
Assume $\bs{\rm H}$, let $f$ be as in~(\ref{unaeffe}b) and suppose $K(r)$ satisfies \textbf{K}.
Set $l_u=l(q,\delta_0)$ and $l_s=l(q,\delta_{\infty})$, and assume $2_*(\eta)<l_u<2^*$,
$2_*(\beta)<l_s<{\rm I}(\beta)$. Then there is an increasing sequence $(B_k)_{k\geq 0}$ such that $v(r,B_{k})$ is a \sol Rkf.
 Moreover, $v(r,L)$
is a  \sol S0f for any $0<L<B_0$, and there is $B^*_k \in [B_{k-1},B_k)$ such that
 $v(r,L)$ is a \sol Skf whenever $B^*_k<L<B_{k}$, for any $k \ge 1$.
Consequently, there is an increasing  sequence $(A_k)_{k\geq 0}$ such that $u(r,A_k)$ is a \sol Rkf for any $k \ge 0$.
\end{cor}

Notice that both the corollaries provide the existence of a positive G.S. with fast decay (the \sol R0f) for equation~\eqref{laplace}. The first one gives also existence of positive G.S. with slow decay (the \sol R0s's) and the second the existence of positive S.G.S. with fast decay (the \sol S0f's).

The previous corollaries focus on solutions which are {\em positive near zero}; by the way similar statements hold for {\em negative near zero} solutions.

In the proofs we apply the classical Fowler transformation, to pass from~\eqref{hardy} to a system,
 and we apply   phase plane analysis and techniques from the
 theory of invariant manifold for non-autonomous
systems, following the way paved by \cite{ JPY1,JPY2,JK}.
Therefore the existence of \sol R{}f corresponds to the existence of homoclinic orbit in
 the introduced dynamical system.
The presence of the Hardy potential  forces us to abandon the classical results established in~\cite{CoLe},
and to add a discussion of   exponential dichotomy  tools, based on~\cite{Co65,Co78,Jsell}.

\medbreak

Kelvin inversion $u(r) \mapsto u(1/r)r^{2-n}$ assumes a particularly clear form when it
 is combined with Fowler transformation (see Section~\ref{SecKelvin}). To the best
 of our knowledge this simple but useful remark   appeared for the first time in~\cite{Fjdde};
 here we explore this fact a bit further.

The paper is structured as follows: in Section~\ref{sec2} we introduce Fowler transformation (§\ref{secfow}), and we explain
some well known correspondences between the new system and the original problem, in the \eqref{eq.na} case (§\ref{sec2.2}), in the \eqref{hardy} case (§\ref{SecMH});
then we state our results in the general framework (§\ref{secST}).
In Section~\ref{prellemmas} we develop some geometrical consideration, which will be actually used in Section~\ref{secprooftm3} to prove our main theorems,
following some ideas introduced in~\cite{BDF,DF,JK}. In Section~\ref{MAP} we give some further examples of application of our results.
In Appendix~\ref{app1} we recall some well known facts concerning invariant manifold theory for non-autonomous system, and we explain
our extension to a setting where continuability is lost. Appendix~\ref{app2} and~\ref{app3} are devoted to adapt to our setting
some topological ideas already used respectively in~\cite{BDF,DF}, and in~\cite{Fduegen}.

\section{Preliminaries and stating of the results.}\label{sec2}
\subsection{Fowler transformation.} \label{secfow}

We consider equation~\eqref{hardy},
which corresponds to radial solutions of~\eqref{Hlaplace}.
Once we have fixed a constant $l>2$, and the values
$$
\alpha_l =\frac{2}{l-2}\,, \quad \gamma_l = \alpha_l+2-n \,,
$$
setting

\begin{equation}\label{transf1}
\begin{cases}
x_l(t)=u(r)r^{\alpha_l} \\  y_l(t) =u'(r)r^{\alpha_l+1}
\end{cases}
\qquad \text{where } r=\eu^t \,,
\end{equation}
 we pass from~(\ref{hardy}) to the following
\begin{equation} \label{si.hardy}\tag{S}
\left( \begin{array}{c}
\dot{x}_l \\
\dot{y}_l  \end{array}\right) = \left( \begin{array}{cc} \alpha_l &
1
\\ -h(\eu^t) & \gamma_l
\end{array} \right)
\left( \begin{array}{c} x_l \\ y_l  \end{array}\right) +\left(
\begin{array}{c} 0 \\-
g_l(x_l,t)\end{array}\right) ,
\end{equation}
where
\begin{equation}\label{whoisgl}
g_l(x,t)=f(x \eu^{-\alpha_l t}, \eu^t) \eu^{(\alpha_l+2)  t}\,.
\end{equation}
In particular, in the classical Laplace case, i.e. when $h(t) \equiv 0$, we find
\begin{equation} \label{si.nagen}\tag{S$_0$}
\left( \begin{array}{c}
\dot{x}_l \\
\dot{y}_l  \end{array}\right) = \left( \begin{array}{cc} \alpha_l &
1
\\ 0 & \gamma_l
\end{array} \right)
\left( \begin{array}{c} x_l \\ y_l  \end{array}\right) +\left(
\begin{array}{c} 0 \\-
g_l(x_l,t)\end{array}\right) .
\end{equation}

\noindent
The main advantage in this change of variables
is that, when $f$ is of type~(\ref{regular}), setting $l=q$ we obtain
a system which is not anymore singular. Moreover, if $h$ and $k$ are constants, then~(\ref{si.hardy})
is autonomous: in fact~\eqref{transf1} is a slight modification of the original transformation introduced by Fowler~\cite{Fow}.

More in general, whenever $h$ is a constant and $k(r)=K r^{\delta}$, where $\delta>-2$ we can set $l=l(q,\delta)=2\frac{q+\delta}{2+\delta}$
to get $g_l(x,t)=K x_l|x_l|^{q-2}$, so that~(\ref{si.hardy}) is an autonomous system.

Assume first that $h \equiv 0$, and $g_l(x,t)=K x_l|x_l|^{q-2}$.
In these cases, whenever $l>2_*$ and $K>0$, the origin is a saddle and admits a
 $1$-dimensional unstable manifold $M^u$ and a $1$-dimensional  stable manifold $M^s$.
 Moreover we have  two critical points $\boldsymbol{P^+} =(P_x,P_y)$ and $\boldsymbol{P^-} =(-P_x,-P_y)$ with $P_x>0$, which are stable for
 $l>2^*$,   centers for $l=2^*$ and unstable for $2_*<l<2^*$.
Using the translation for this context of the Pohozaev identity, see e.g.~\cite{Poho},
 we can easily
draw the phase portrait, in such an autonomous case (a detailed proof in the $p$-Laplace
context is given in~\cite{F2}, see also~\cite{Fjdde,FmSa}).

If $K<0$ and $l>2_*$ the origin is the unique critical point and both $M^u$ and $M^s$ are unbounded curves, see Figure \ref{livelli}.

\begin{figure}[t!]
\centering
\centerline{\epsfig{file=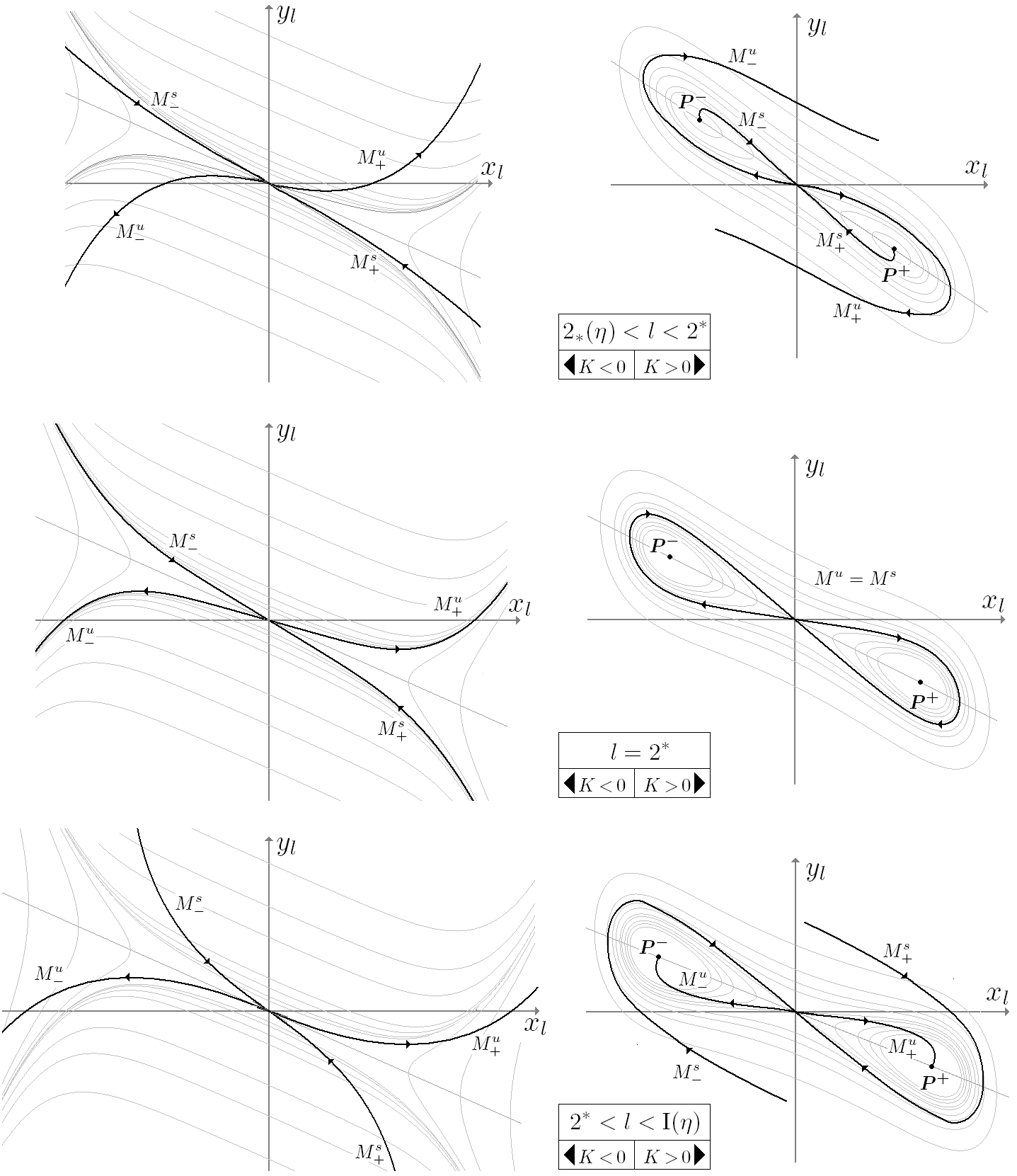, width = 12 cm}}
  \caption{
  Sketch of the phase portrait of~(\ref{si.hardy}), when $g_l(x,t)$ is $t$-independent and
  satisfies $\GP$. The unstable manifolds $M^u$ and the stable manifolds $M^s$ are drawn.  The manifolds can be located using  the level curves of some energy functions (cf.~\cite{F2,FmSa,Poho}): in particular in the case $l=2^*$ the system is Hamiltonian, moreover, if $K>0$, it presents periodic solutions and $M^u$ and $M^s$ coincide giving two homoclinic orbits.
}\label{livelli}
\end{figure}

In fact this analysis is easily extended to any autonomous system~(\ref{si.nagen}) satisfying the following assumption,
see~\cite{DF} for details:
\begin{description}
  \item[$\GP$] There is $l>2_*$
   such that $g_l(x,t) \equiv K g_l(x)$ is $t$-independent, $K \ne 0$ is a constant, and $g_l(x)/x$ is a function, which is positive increasing for $x>0$ and positive decreasing for $x<0$, satisfying
  $$  \lim_{x \to 0} \frac{g_l(x)}{x} = 0%< |\alpha_l \gamma_l|
  \qquad \text{and}\qquad
  \lim_{|x| \to +\infty} \frac{g_l(x)}{x}=+\infty\,. $$
\end{description}
\noindent
This way we can consider e.g. $g_l(x)=
k_1  x|x|^{q_1-2}+k_2 x|x|^{q_2-2}$
(i.e. $f(u,r)= k_1 u|u|^{q_1-2}+k_2 r^{\delta}u|u|^{q_2-2}$
where $\delta=\frac{2(q_2-q_1)}{q_1-2}$),
 or $g_l(x)= k_1  x|x|^{q-2}\ln(|x|)$ (i.e.
 $f(u,r)= k_1 u|u|^{q-2} \ln(u r^{\frac{2}{q-2}})$)
where $k_1$ and $k_2$ are positive constants
and $q_1$ and $q_2$ are larger than $2$. So we can consider slightly
more general functions $f$.

\medbreak

Now we introduce  a further notation which will be in force in the whole paper: we denote by $\boldsymbol{x_l}(t,\tau,\Q)=(x_l(t,\tau,\Q),y_l(t,\tau,\Q))$
 the trajectory  of~\eqref{si.hardy}  -- or~\eqref{si.nagen} --
which is in $\Q$ for $t=\tau$.

The following remark underlines the relations between the behaviour of solutions of~(\ref{si.nagen}) and of~(\ref{eq.na}).

 \begin{remark}\label{corrisp}
 Assume $\GP$.
Consider the trajectory $\boldsymbol{x_l}(t,\tau,\Q)$ of~(\ref{si.nagen}) and let $u(r)$ be the corresponding
solution of~(\ref{eq.na});
then $u(r)$ is a regular solution if and only if $\Q \in M^u$, while it has fast decay if and only if
$\Q \in M^s$.
 \end{remark}
This result can be easily proved using standard tools in invariant manifold theory, see
e.g.~\cite{F2,Fjdde,FmSa};
we will prove it as a special case of a more general result, Lemma~\ref{corrispondenze}
below, in the non-autonomous context.

Assume now  $h(t) \equiv \eta<\frac{(n-2)^2}{4}$: the linearization of  system~\eqref{si.hardy} in the origin
has real and distinct eigenvalues.
Moreover  the origin is a saddle iff $|\alpha_l \gamma_l|>\eta$ or equivalently
\begin{equation}\label{lh}
2_*(\eta)<l<{\rm I}(\eta)\,,
\end{equation}
where
$2_*(\eta)$ and
${\rm I}(\eta)$ have been defined in~\eqref{def2eta},~\eqref{defIeta}.
The eigenvalues are $\lambda_1 =\gamma_l + \kappa(\eta) <0<\lambda_2=\alpha_l -\kappa(\eta)$, where $\kappa(\eta)$ was defined in~\eqref{defKappa}.
Let us assume first that $g_l(x,t) = K g_l(x)$ satisfies $\GP$ with $K>0$, where the condition $l>2_*$ is replaced by $l>2_*(\eta)$.
It is straightforward to check that, when the parameters are in the range
\eqref{lh},
 we have again a unique critical point $\bs{P^{+}}=(P_x, P_y)$ in $x>0$; in particular  if $g_l(x)=x|x|^{q-2}$,
we find $P_x=[\alpha_l (n-2-\alpha_l)-\eta]^{1/(q-2)}$, and $P_y=-\alpha_l P_x$.
The point $\bs{P^{+}}$ is unstable (either a node or a focus) if $2_*(\eta)<l<2^*$, a center if $l=2^*$ it is stable if $2^*<l<{\rm I}(\eta)$ (either a node or a focus). Again for $K<0$ we find that $M^u$ and $M^s$ are unbounded curves. See Figure \ref{livelli}.
 We refer to~\cite{FmSa} for details.

We can again consider the stable manifold $M^s$ and the unstable manifold $M^u$, in order to obtain estimates as in Remark~\ref{corrisp}, however the presence of the Hardy potential may forbid the existence of regular solutions. We will present such details in Section~\ref{SecMH} in the non-autonomous case so to avoid repetitions.

\subsection{Stable and unstable manifolds for non-autonomous systems.}\label{sec2.2}

In the previous subsection we have begun from the autonomous case for illustrative purposes.
Now we turn to consider the $t$-dependent case: the first step is the generalization of
Remark~\ref{corrisp}.
In this subsection we will make use of the following assumption
for illustrative purposes (it will be removed from the next subsection):
\begin{description}
\item[\textbf{C}] All the trajectories of~\eqref{si.hardy} are
continuable for any $t \in \RR$.
\end{description}

We have two different alternatives to introduce stable and unstable sets for non-autonomous systems,
 thus extending Remark~\ref{corrisp} to a generic
$g_l(x,t)$. The simplest one requires the strongest hypotheses, but gives more structure.

\begin{description}
 \item[$\Gupiu$] Assume  \textbf{H} and that there is $l_u \in (2_*(\eta), {\rm I}(\eta))$ such that
$g_{l_u}(0,t)=0$, $\partial_x g_{l_u}(0,t) = 0$ for any $t\in\RR$,
 and
 $$
 \lim_{t\to -\infty} g_{l_u}(x,t) = K g^{-\infty}_{l_u}(x)
 \quad\text{and}\quad
 \lito \eu^{-\varpi t} \,{\partial_t} \,g_{l_u}(x,t) =0
 \,,
 $$
 uniformly on compact sets, where the function $g^{-\infty}_{l_u}$ is a non-trivial locally Lipschitz function satisfying $\GP$ and $\varpi$ is a suitable positive constant.
\item[$\Gspiu$] Assume  \textbf{H} and that there is $l_s \in (2_*(\beta), {\rm I}(\beta))$ such that
$g_{l_s}(0,t)=0$, $\partial_x g_{l_s}(0,t) = 0$  for any $t \in \RR$, and
  $$
  \lim_{t\to +\infty} g_{l_s}(x,t) = K g^{+\infty}_{l_s}(x)
  \quad\text{and}\quad
  \lim_{t\to +\infty} \eu^{\varpi t} \, {\partial_t} \, g_{l_s}(x,t)=0
  \,,
  $$
  uniformly on compact sets, where the function $g^{+\infty}_{l_s}$ is a non-trivial locally Lipschitz function satisfying $\GP$ and $\varpi$ is a suitable positive constant.
  \end{description}

\medbreak

\noindent  Assume $\Gupiu$ and  add to~(\ref{si.hardy}) the variable $z= \eu^{\varpi t}$, to get

\begin{equation}\label{si.naa}
\left( \begin{array}{c}
\dot{x}_{l_u} \\
\dot{y}_{l_u} \\
\dot{z} \end{array}\right) = \left( \begin{array}{ccc} \alpha_{l_u} &
1 &0
\\ -h( z^{1/\varpi}) & \gamma_{l_u} & 0 \\
0 & 0 & \varpi
\end{array} \right)
\left( \begin{array}{c} x_{l_u} \\ y_{l_u} \\ z \end{array}\right) -\left(
\begin{array}{c} 0 \\
g_{l_u}(x_{l_u},\ln(z)/\varpi)\\ 0\end{array}\right) .
\end{equation}
We have thus obtained an autonomous system and all its
 trajectories converge to the $z=0$ plane as $t \to -\infty$; so~(\ref{si.naa}) is useful to investigate
the asymptotic behavior in the past.
 The origin admits a
$2$-dimensional  unstable manifold denoted by $\boldsymbol{W^u}$.
  From standard arguments of dynamical system
 theory, we see that  the set
 $W^u_{l_u}(\tau)=\boldsymbol{W^u} \cap \{ z= \eu^{\varpi \tau} \}$ is a $1$-dimensional (immersed) manifold, for any $\tau \in \RR$,
 see e.g.~\cite{BDF,Fdie}.

  Similarly when $\Gspiu$ holds we   consider the following  system to study the behavior of
  trajectories in the future, adding the new variable $\zeta = \eu^{-\varpi t}$.
 \begin{equation}\label{si.naas}
\left( \begin{array}{c}
\dot{x}_{l_s} \\
\dot{y}_{l_s} \\
\dot{\zeta} \end{array}\right) = \left( \begin{array}{ccc} \alpha_{l_s} &
1 &0
\\ -h( \zeta^{-1/\varpi}) & \gamma_{l_s} & 0 \\
0 & 0 & -\varpi
\end{array} \right)
\left( \begin{array}{c} x_{l_s} \\ y_{l_s} \\ \zeta
\end{array}\right) -\left(
\begin{array}{c} 0 \\
g_{l_s}(x_{l_s},-{\ln(\zeta)}/{\varpi})\\ 0\end{array}\right) .
\end{equation}
All its trajectories converge to the $\zeta=0$ plane as $t \to +\infty$.  The origin admits a
 $2$-dimensional stable manifold denoted by $\boldsymbol{W^s}$. Arguing as above, for any $\tau \in \RR$,
 $W^s_{l_s}(\tau)=\boldsymbol{W^s} \cap \{ \zeta= \eu^{-\varpi \tau} \}$ is a $1$-dimensional manifold.

 \medbreak

 Let us denote by $W^u_{l_u}(-\infty)$ the unstable manifold $M^u$ of the autonomous system~(\ref{si.hardy}) where
  $l=l_u$ and  $g_{l_u}(x,t) \equiv K g_{l_u}^{-\infty}(x)$,  and by  $W^s_{l_s}(+\infty)$ the
  stable manifold $M^s$ of the autonomous system~(\ref{si.hardy}) where
  $l=l_s$ and  $g_{l_s}(x,t) \equiv K g_{l_s}^{+\infty}(x)$.  Then we have the following.

    \begin{remark}\label{allinfinito}
  Assume \textbf{C} and $\Gupiu$; then $W^u_{l_u}(\tau)$ approaches
  $W^u_{l_u}(-\infty)$ as $\tau \to -\infty$. Assume $\Gspiu$; then $W^s_{l_s}(\tau)$ approaches
  $W^s_{l_s}(+\infty)$ as $\tau \to +\infty$.
  More precisely, if
 $W^u_{l_u}(\tau_0)$ (respectively
  $W^s_{l_s}(\tau_0)$) intersects  transversally a certain line $L$ in a point $\Q(\tau_0)$ for $\tau_0 \in[-\infty,+\infty)$
  (resp. for $\tau_0 \in(-\infty,+\infty]$), then there is a neighborhood
  $I$  of $\tau_0$ such that $W^u_{l_u}(\tau)$ (resp.
  $W^s_{l_s}(\tau)$) intersects $L$ in a point $\Q(\tau)$ for any $\tau \in I$, and $\Q(\tau)$ is continuous
  (in particular it is as smooth as $g_l$).
  \end{remark}

  The proof of this Remark follows from standard facts in dynamical system theory. If we assume $h(t)\equiv \eta$, it follows from
    ~\cite[§ 13]{CoLe}, while if we allow $h$ to be a function satisfying \textbf{H} it follows from
    ~\cite[Theorem 4.1]{Co78}  or~\cite[Theorem~2.16]{Jsell}.

  Following~\cite{JPY2}, which is based on~\cite{Jsell}, we can introduce stable and unstable leaves with assumptions
weaker than $\Gupiu$ and $\Gspiu$, see also the Appendix~\ref{app1} for a more detailed discussion of the topic.
\begin{description}
\item[$\Gumeno$] Assume \textbf{H} and that there exists $l_u \in (2_*(\eta), {\rm I}(\eta))$ such that
$g_{l_u}(x,t)$ is  continuous in $x$ uniformly for $t \le \tau$,
whenever  $\tau \in\RR$ and for any  $x$ in a compact set; further
$g_{l_u}(0,t)=\partial_x g_{l_u}(0,t) = 0$ for any $t \in \RR$.
\item[$\Gsmeno$] Assume \textbf{H} and that there exists $l_s \in (2_*(\beta), {\rm I}(\beta))$ such that
$g_{l_s}(x,t)$ is  continuous in $x$ uniformly for $t \ge \tau$,
whenever $\tau \in\RR$ and for any $x$ in a compact set; further
$g_{l_s}(0,t)=\partial_x g_{l_s}(0,t) = 0
$ for any $t \in \RR$.
\end{description}
Replacing $\Gupiu$ by $\Gumeno$ and $\Gspiu$ by $\Gsmeno$
  we can again construct $1$-dimensional (immersed) manifolds $W^u_{l_u}(\tau)$,
  respectively $W^s_{l_s}(\tau)$, for any $\tau \in\RR$ by characterizing them as follows:
    \begin{equation}\label{WueWs}
        \begin{array}{c}
    W^u_{l_u}(\tau):= \{ \Q \in \RR^2 \mid \lito \boldsymbol{x_{l_u}}(t,\tau, \Q)=(0,0) \} \,,\\
    W^s_{l_s}(\tau):= \{ \Q \in \RR^2 \mid \lit  \boldsymbol{x_{l_s}}(t,\tau, \Q)=(0,0) \} \,.
    \end{array}
  \end{equation}
  Furthermore $ W^u_{l_u}(\tau)$ and $ W^s_{l_s}(\tau)$  in the origin are tangent respectively
  to the unstable and the stable space of the linearized system, see~\cite{Jsell} and the Appendix for more details, in particular Remarks~\ref{global} and~\ref{8shaped}.

  $ W^u_{l_u}(\tau)$ and $ W^s_{l_s}(\tau)$ have the smoothness property described above in Remark~\ref{allinfinito}, but the first part of Remark~\ref{allinfinito} concerning their asymptotical behavior does not
   hold anymore (since $W^u_{l_u}(-\infty)$ and $W^s_{l_s}(+\infty)$ may be not defined).
   We stress that $\Gupiu$ implies $\Gumeno$, and   if the former holds then the manifolds $W^u_{l_u}(\tau)$
   constructed via $\Gupiu$ and $\Gumeno$ coincide;
    the specular result holds  for $\Gspiu$ which implies $\Gsmeno$
   (this way we see that Remark~\ref{tangente} below holds for $ W^u_{l_u}(\tau)$ and $ W^s_{l_s}(\tau)$ if we assume $\Gupiu$ and $\Gspiu$).
   However observe that with $\Gumeno$ and $\Gsmeno$ we allow also functions $g_l(x,t)$ which are periodic in $t$ or which have
   logarithmic behavior.
 Further notice that when $\Gupiu$ holds,
   the phase portrait is very different in the two cases $K>0$ and $K<0$ (see Figure~\ref{livelli}), but in any case $\Gumeno$ holds and guarantees the existence
   of the unstable manifold. A similar argument holds for $\Gspiu$ and $\Gsmeno$, too.

   Since we want to
  understand the mutual position of these two objects we introduce the manifolds:
  \begin{equation}\label{cambioL}
    \begin{array}{l}
      W^u_{l_s}(\tau):= \{ \R=\Q \textrm{exp}[-(\alpha_{l_u}-\alpha_{l_s})\tau] \in \RR^2 \mid \Q \in W^u_{l_u}(\tau) \} \,,\\
      W^s_{l_u}(\tau):= \{ \Q=\R \textrm{exp}[(\alpha_{l_u}-\alpha_{l_s})\tau] \in \RR^2 \mid \R \in W^s_{l_s}(\tau) \} \,.
    \end{array}
  \end{equation}
  Notice that $W^u_{l_s}(\tau)$ is omothetic to $W^u_{l_u}(\tau)$, and $ W^s_{l_u}(\tau)$ is
  omothetic to $W^s_{l_s}(\tau)$. Hence, they are $1$-dimensional (immersed) manifolds for any $\tau \in \RR$.

\begin{remark}\label{cambialu}
 Observe that when $\Humeno$ holds for some $\bar{l}_u>2_*(\eta)$, respectively $\Hsmeno$ holds for some
  $\bar{l}_s>2_*(\beta)$, then
$\Humeno$ holds for any $l_u \in[ \bar{l}_u, {\rm I}(\eta))$, resp. $\Hsmeno$ holds for any
$l_s \in  (2_{*}(\beta),\bar{l}_s]$.
\end{remark}
The validity of the previous remark can be immediately verified: if we choose $l\neq L$ we have
$g_L(x,t)/x = g_l(\xi,t)/\xi$, where $ \xi= x \eu^{(\alpha_l-\alpha_L)t} $.

\medbreak

  \textbf{From now to the end of the subsection we assume $\bs{h(t) \equiv 0}$} for illustrative purposes; such a restriction is removed in the next
  subsection, which is focused on the novelties introduced by the Hardy term.
  We emphasize that, in any case, the result of this paper are new even for the original Laplace case, i.e. when $h(t) \equiv 0$.
  We also stress that the upper bound in the values of $l_u$ and $l_s$ due to $\Gumeno$, $\Gsmeno$ disappears when $h(t) \equiv 0$
  (since ${\rm I}(0)=+\infty$).

  \begin{remark}\label{tangente}
  Assume $h(t) \equiv 0$, \textbf{C}, $\Gumeno$, $\Gsmeno$, then $ W^u_{l_u}(\tau)$ and $ W^s_{l_s}(\tau)$ are tangent respectively to
  $y=0$ and to $y=-(n-2)x$ at the origin for any $\tau \in\RR$.
  \end{remark}

   As in the $t$-independent case,  all regular solutions correspond to trajectories in $W^u_{l_s}(\tau)$,
  while fast decay solutions correspond  to trajectories in $W^s_{l_u}(\tau)$.
 More precisely   we have the following, see~\cite{Fjdde,Fdie}.

  \begin{lemma}\label{corrispondenze}
Assume $h \equiv 0$, \textbf{C}, $\Gumeno$ and  $\Gsmeno$. Consider the trajectory $\boldsymbol{x_{l_u}}(t,\tau,\Q)$ of~(\ref{si.nagen}) with $l=l_u>2_*$, and the corresponding trajectory
$\boldsymbol{x_{l_s}}(t,\tau,\R)$ of~(\ref{si.nagen}) with $l=l_s>2_*$. Let $u(r)$ be the corresponding
solution of~(\ref{eq.na}). Then $\R=\Q \textrm{exp}[ -(\alpha_{l_u}-\alpha_{l_s}) \tau]$,
\begin{equation*}
\begin{aligned}
	u(r) \text{ is a regular solution}
	&\,\iff\, &
  \Q \in W^u_{l_u}(\tau)
 & \,\iff\, &
  \R \in W^u_{l_s}(\tau)\,,
\\
	u(r) \text{ is a fast decay  solution}
	& \,\iff\, &
  \Q \in W^s_{l_u}(\tau)
  & \,\iff\, &
  \R \in W^s_{l_s}(\tau) \,.
\end{aligned}
\end{equation*}
\end{lemma}

\noindent
We postpone the proof of the lemma to Appendix~\ref{app1}, see page \pageref{proofcorr}.

\medbreak

 Note that the manifold $W^u_{l_u}(\tau)$ is split by the origin into two connected components, one which
 leaves the origin and enters the $x>0$ semi-plane (corresponding to regular solutions $u(r)$ positive for $r$ small), denoted by
 $W^{u,+}_{l_u}(\tau)$, and the other which enters the $x<0$ semi-plane (corresponding to regular solutions $u(r)$
 negative for $r$ small), denoted by
 $W^{u,-}_{l_u}(\tau)$. Similarly $W^{s}_{l_s}(\tau)$ is split by the origin into $W^{s,+}_{l_s}(\tau)$ and $W^{s,-}_{l_s}(\tau)$,
 which leave the origin and enter respectively in $x>0$ and in $x<0$ (and correspond to fast decay solutions $u(r)$
 which are definitively positive and definitively negative respectively).

\medbreak

Now we turn to consider briefly singular and slow decay solutions, see e.g.~\cite{DF}.

\begin{remark}\label{sing}
Assume $h(t) \equiv 0$ and  $\Gupiu$ with $K>0$, then~(\ref{si.naa}) admits a critical point
$(P_x,P_y,0)$ such that $P_x>0$. This point admits an unstable manifold which is $1$-dimensional if $l_u \ge 2^*$
and $3$-dimensional if $2_*<l_u <2^*$. The trajectories $(\boldsymbol{x_{l_u}}(t),z(t))$ contained in this manifold
 correspond to singular solutions $v(r)$ of~(\ref{eq.na})
such that $\liro v(r)r^{\alpha_{l_u}}=P_x>0$. It is easy to check that  if $l_u>2^*$ we have a unique singular solution,
while if $2_*<l_u<2^*$ we have uncountably many singular solutions.
Further, if $l_u \ne 2^*$, any trajectory $\bs{x_{l_u}}(t)$ of \eqref{si.hardy} bounded for $t \le 0$ converges
  either to $\bs P= (P_x,P_y)$ or to $-\bs P$ or to the origin as $t \to -\infty$.

Assume $h(t) \equiv 0$ and   $\Gspiu$ with $K>0$, then~(\ref{si.naas}) admits a critical point
$(P_x,P_y,0)$ such that $P_x>0$. This point admits a
stable manifold which is $1$-dimensional  if $2_*< l_s \le 2^*$
and $3$-dimensional if $l_s>2^*$.
 The trajectories $(\boldsymbol{x_{l_s}}(t),\zeta(t))$ contained in this manifold
 correspond to slow decay solutions $v(r)$ of~(\ref{eq.na})
such that $\lir v(r)r^{\alpha_{l_s}}=P_x>0$.
If $2_*<l_s<2^*$ we have a unique slow decay solution,
while if $l_s>2^*$ we have uncountably many slow decay solutions.
Further, if $l_s \ne 2^*$, any trajectory $\bs{x_{l_s}}(t)$ of \eqref{si.hardy} bounded for $t \ge 0$ converges
either to $\bs P= (P_x,P_y)$ or to $-\bs P$ or to the origin as $t \to +\infty$.
\end{remark}
The proof follows from elementary arguments on the phase portrait, see e.g. \cite[Lemma 2.9]{DF} for details.

\subsection{Stable and unstable manifolds with
Hardy potentials.}\label{SecMH}
We go back to consider~\eqref{hardy}, and~\eqref{si.hardy} where $h(t) \not\equiv 0$ satisfies \textbf{H}.
We list some results which explain similarities and   differences with respect to Section 2.2.
Their proofs rely on standard facts in invariant manifold theory for non-autonomous systems, and in particular on exponential dichotomy: they
are postponed to the Appendix.
\begin{remark}\label{differenze1}
 Assume   $\Humeno$; if $\eta \ne 0$ regular solutions for~\eqref{hardy} do not exist, due to the singularity of the equation for $r=0$.
They are replaced by solutions which (may) exhibit a singular behavior as $r \to 0$. More precisely,
for any $d \in \RR$ there is a unique solution, $u(r)=u(r,d)$ of~\eqref{hardy} such that
$u(r)r^{\kappa(\eta)} \to d$ as $r \to 0$. \\
Analogously assume $\Hsmeno$; then the behavior of fast decay solutions changes slightly. I.e., for any $L \in \RR$ there is a unique
solution $v(r,L)$ such that $v(r,L)r^{n-2-\kappa(\beta)} \to L$ as $r \to +\infty$ (cf. Definition~\ref{defRFS}).

In particular, Lemma~\ref{corrispondenze} continues to hold respectively for \Rsol-solutions and \fdsol-solutions.
\end{remark}
\begin{remark}\label{differenze1.5}
 Assume   $\Gupiu$ and $\Gspiu$ with $K>0$ and \textbf{H} (allow $h(t)\not\equiv 0$). Then Remark~\ref{sing}, continues to hold almost with no differences:
 In particular  \Ssol-solutions are  asymptotic to $\pm P_x r^{\alpha_{l_u}}$ as $r \to 0$, while \sdsol-solutions  are asymptotic to
 $\pm P_x r^{\alpha_{l_s}}$ as $r \to +\infty$.
  The only change is in the value
 of $P_x$ (which however can be computed explicitly).

 We also observe that if $\Gupiu$, $\Gspiu$, \textbf{H} hold but $K<0$, then there are no \Ssol-solutions neither \sdsol-solutions, so solutions of \eqref{hardy} which are defined in a neighbourhood of $r=0$ and are definitively positive  for $r$ small are \Rsol-solutions and the ones which are defined in a neighbourhood of $r=\infty$ and are definitively positive  for $r$ large are \fdsol-solutions.
\end{remark}

Denote by $\mathcal{A}_l(t)=\left( \begin{array}{cc} \alpha_{l} &
1
\\ -h(\eu^t) & \gamma_{l}
\end{array} \right)$.
We recall that if $\mathcal A_l(t) \equiv \mathcal{A}_l$ is a constant matrix (e.g. when
 $h \equiv 0$ as for~\eqref{eq.na}), then the tangent spaces to $W^u_{l_u}(\tau)$ and $W^s_{l_s}(\tau)$,
 say $\ell^u(\tau)$ and $\ell^s(\tau)$,
 are independent from $\tau$. This is not the case if $\mathcal{A}_l(t) \not\equiv \mathcal{A}$.
 Let  $m^u(\tau)$  and $m^s(\tau)$  be such that
 \begin{equation}\label{ellus}
 \begin{array}{cc}
    \ell^u(\tau):=\{(1, m^u(\tau))  \mid x \in \RR \} \, , &   \ell^u(-\infty):=\{(1, m^u(-\infty)) x  \mid x \in \RR \} \, , \\
      \ell^s(\tau):=\{(1, m^s(\tau)) x  \mid x \in \RR \} \,,
       &  \ell^s(+\infty):=\{(1, m^s(+\infty))x   \mid x \in \RR \} \,.
    \end{array}
 \end{equation}
\begin{remark}\label{differenze2}
Assume $\Humeno$, $\Hsmeno$  and allow $h(t)\not\equiv 0$; then  $\ell^u(\tau)$ and $\ell^s(\tau)$, change smoothly with $\tau$.
Moreover $m^u(\tau)\to m^u(-\infty):=-(\kappa(\eta))$  as $\tau \to -\infty$ and
$m^s(\tau)\to m^s(+\infty):=-(n-2-\kappa(\beta))$  as $\tau \to +\infty$.
Furthermore
$$
\kappa(\eta) < \frac{n-2}{2} < n-2-\kappa(\beta) \,.
$$
\end{remark}
\begin{remark}\label{differenze3}
If $0<\eta<\frac{(n-2)^2}{4}$, then  $\kappa(\eta)>0$, hence the \Rsol-solution $u(r,d)$, with $d>0$,
 is in fact singular, i.e. $\liro u(r)=+\infty$, and accordingly $u'(r)$ is negative and $\liro u'(r)=-\infty$
 as $r \to 0$. However if $\eta<0$ then $\kappa(\eta)<0$, i.e. the \Rsol-solution $u(r,d)$, with $d>0$, is such that $u(r,d) \to 0$ like a power
 as $r \to 0$. Moreover it is monotone increasing for $d>0$, since $\ell^u(\tau)$
 lies in  $xy>0$ for $\tau\ll0$, and consequently the first branch of $W^u_{l_u}(\tau)$
 lies in $xy>0$ for $\tau\ll0$.
\end{remark}

\subsection{The lack of continuability}\label{sec2.4}

If \textbf{C} is removed the situation becomes more complicated.
\begin{remark}\label{continuabilita}
In this paper we are interested in functions $f(u,r)$ which are negative for $u$ large and either $r$ small or $r$ large.
  In these cases equation~\eqref{laplace} may admit solutions which are not globally defined,
  i.e. \textbf{C} is not fulfilled.
   So we adopt the following notation: we   say that a solution $u(r,d)$, resp. a solution $v(r,L)$, is defined in a certain maximal interval $[0,\varrho_d)$ with $\varrho_d \in (0,+\infty]$, resp. in $(\bar{\varrho}_L,+\infty)$ with $\bar{\varrho}_L\in[0,+\infty)$.
   We emphasize that $\varrho_d \to +\infty$ as $d \to 0$, and $\bar{\varrho}_L \to 0$ as $L \to 0$,
   since the null solution is continuable for any $r \ge 0$.
\end{remark}
We introduce the following definitions
\begin{equation}\label{eqrhod}
    \begin{array}{ccc}
      d^+_\tau & =&  \sup \{D \mid  \rho_d < \eu^{\tau} \; \textrm{ for any $0<d< D$} \}\,, \\
      d^-_\tau & =&  \inf \{D \mid  \rho_d < \eu^{\tau} \; \textrm{ for any $D<d<0$} \}\,, \\
     L^+_\tau & =&  \sup \{\bar{L} \mid  \bar{\rho}_{L} > \eu^{\tau} \; \textrm{ for any $0<L< \bar{L}$} \}\,,\\
       L^-_\tau & =&  \inf \{\bar{L} \mid  \bar{\rho}_{L} > \eu^{\tau} \; \textrm{ for any $\bar{L}<L<0$} \}\,,
    \end{array}
\end{equation}
Obviously the intervals $(d^-_\tau,d^+_\tau)$ and  $(L^-_\tau,L^+_\tau)$ coincide with the whole $\RR$ if \textbf{C}
is assumed, but they are bounded in the cases considered in this paper.

The lack of continuability is a relevant problem in order to apply
dynamical system techniques and invariant manifold theory for non-autonomous systems.
Let us denote by $\tilde{W}^u_{l_u}(\tau)$
and $\tilde{W}^s_{l_s}(\tau)$ the sets characterized as in~\eqref{WueWs}.
In the Appendix we will show that $\tilde{W}^u_{l_u}(\tau)$
and $\tilde{W}^s_{l_s}(\tau)$   may be disconnected. As usual, we can split these sets in their components $\tilde{W}^{u,\pm}_{l_u}(\tau)$
and $\tilde{W}^{s,\pm}_{l_s}(\tau)$, which may be disconnected too.

For any fixed $\tau$, let us denote by $W^u_{l_u}(\tau)$ and $W^s_{l_s}(\tau)$
respectively
the connected component of $\tilde{W}^u_{l_u}(\tau)$
and $\tilde{W}^s_{l_s}(\tau)$ containing the origin, which are $1$-dimensional manifolds, as we will
see just below. We stress that there is no abuse of notation since,
if \textbf{C} is assumed, $\tilde{W}^u_{l_u}(\tau)$
and $\tilde{W}^s_{l_s}(\tau)$ are $1$-dimensional connected manifolds so they coincide with $W^u_{l_u}(\tau)$ and $W^s_{l_s}(\tau)$
respectively. Similarly, we can introduce the connected branches departing from the origin ${W}^{u,\pm}_{l_u}(\tau)\subset \tilde{W}^{u,\pm}_{l_u}(\tau)$
and ${W}^{s,\pm}_{l_s}(\tau) \subset \tilde{W}^{s,\pm}_{l_s}(\tau)$ without abuse of notation.

\begin{lemma}\label{corrW}
Assume   $\Gumeno$ and  $\Gsmeno$, then there are $1$-dimensional  immersed manifolds $W^u_{l_u}(\tau)$ and $W^s_{l_s}(\tau)$
with the following properties: they contain the origin, they are connected, and they are subsets of the
sets $\tilde{W}^u_{l_u}(\tau)$
and $\tilde{W}^s_{l_s}(\tau)$ characterized as in~\eqref{WueWs}.
\end{lemma}
Lemma~\ref{corrispondenze} continues to hold with the following changes.

  \begin{lemma}\label{corrispondenze.bis}
Assume $\Gumeno$ and  $\Gsmeno$.
Consider the trajectory $\boldsymbol{x_{l_u}}(t,\tau,\Q)$ of~(\ref{si.hardy}) with $l=l_u \in (2_*(\eta), {\rm I}(\eta))$, and
the corresponding trajectory $\boldsymbol{x_{l_s}}(t,\tau,\R)$ of~(\ref{si.hardy}) with $l=l_s\in (2_*(\beta), {\rm I}(\beta))$.
Let $u(r)$ be the corresponding solution of~(\ref{hardy}).
 Then $\R=\Q \textrm{exp}[ -(\alpha_{l_u}-\alpha_{l_s}) \tau]$.
\begin{equation*}
\begin{aligned}
	u(r) \text{ is a \Rsol-solution}
	&\,\iff\, &
  \Q \in \tilde W^u_{l_u}(\tau)
 & \,\iff\, &
  \R \in \tilde W^u_{l_s}(\tau)\,,
\\
	u(r) \text{ is a \fdsol-solution}
	& \,\iff\, &
  \Q \in \tilde W^s_{l_u}(\tau)
  & \,\iff\, &
  \R \in \tilde W^s_{l_s}(\tau) \,.
\end{aligned}
\end{equation*}
Recall that $W^u_{l_u}(\tau)\subset \tilde{W}^u_{l_u}(\tau)$, $W^u_{l_s}(\tau)\subset \tilde{W}^u_{l_s}(\tau)$.
Consider a \Rsol-solution $u(r)$, then we can find $N \gg 0$ such that $\Q \in W^u_{l_u}(\tau)$ and $\R \in W^u_{l_s}(\tau)$
for any $\tau <-N$. Similarly consider a \fdsol-solution $u(r)$, then we can find $N \gg 0$ such that $\Q \in W^s_{l_u}(\tau)$ and $\R \in W^s_{l_s}(\tau)$
for any $\tau >N$.
\end{lemma}

The proofs of the previous lemmas are postponed to Appendix~\ref{app1}.

\begin{remark}\label{tangentiprol}
  We stress that the tangents to
  $\tilde{W}^u_{l_u}(\tau)$ and to
  $\tilde{W}^s_{l_s}(\tau)$ in the origin coincide by construction with the
  tangents to $W^u_{l_u}(\tau)$ and to
  $W^s_{l_s}(\tau)$. Hence, in Remark~\ref{tangente} we can remove assumption
  \textbf{C}.
  \end{remark}

\subsection{Kelvin inversion and Fowler transformation}\label{SecKelvin}

  An important tool in the investigation of
 equations like~(\ref{hardy})
  is a change of variables known as Kelvin inversion.
  Let us set
\begin{equation}\label{kelvinvero}
\begin{split}
  s=1/r \, , \quad  \tilde{u}(s)=u(1/s) s^{2-n}\, ,\quad  \tilde{f}(\tilde{u},s)=
  f(\tilde{u} s^{n-2},1/s)s^{-2-n} \, .
\end{split}
\end{equation}
From a straightforward computation we see that  $u(r)$ satisfies~(\ref{hardy})
if and only if  $\tilde{u}(s)$ satisfies the following equation:
\begin{equation}\label{kelvinvero1}
\frac{d}{ds} [\tilde{u}_{s}(s)s^{n-1}] + \tilde{f}(\tilde{u},s)s^{n-1}=0\, .
\end{equation}
In particular, if $f$ is of type~(\ref{regular}) then $\tilde{f}(u,s)= k(1/s) s^{(n-2)(q-2^*)} u|u|^{q-2}$.

We stress that  \Rsol-solutions $u(r,d)$ of~(\ref{hardy}) are driven by
(\ref{kelvinvero}) into \fdsol-solutions $\tilde{v}(s,d)=u(1/s,d)s^{2-n}$
of~(\ref{kelvinvero1}), while
\fdsol-solutions $v(r,L)$ of~(\ref{hardy}) are driven
into \Rsol-solutions $\tilde{u}(s,L)=v(1/s,L)s^{2-n}$
of~(\ref{kelvinvero1}); we emphasize that $d=\liro u(r)r^{\kappa(\eta)}=\lim_{s \to +\infty} \tilde{v}(s)s^{n-2-\kappa(\eta)}$,
and $L=\lim_{r \to +\infty} v(r)r^{n-2-\kappa(\beta)}=\lim_{s \to 0}\tilde{u}(s)s^{n-2-\kappa(\beta)}$.
It is important to observe that generically if $f$ is subcritical (respectively supercritical)
then $\tilde{f}$ is supercritical (resp. subcritical), see, e.g., \cite[§2]{Fjdde} for the analogous statement for \eqref{eq.na}.

\medbreak

We emphasize that  Kelvin inversion~(\ref{kelvinvero})
 assumes a more clear form when it is
combined with  Fowler transformation~(\ref{transf1}).
In fact, when we  apply~(\ref{transf1}) to~(\ref{kelvinvero1}), by setting
\begin{equation}\label{transf2}
\begin{cases}
  \tau=-t \\
  \tilde{x}_l(\tau)=\tilde{u}(\eu^{\tau}) \eu^{-\gamma_l \tau}=
  u(\eu^{-\tau})\eu^{-\alpha_l \tau}= u(\eu^{t})\eu^{\alpha_l t} \\
  \tilde{y}_l(\tau)=\tilde{u}'(\eu^{\tau}) \eu^{(-\gamma_l+1) \tau}
   = -u'(\eu^{t})\eu^{(\alpha_l+1)t}-(n-2)u(\eu^{t})\eu^{\alpha_l t}
  \end{cases}
\end{equation}
we simply pass from system~(\ref{si.hardy}) to the following one:
\begin{equation} \label{si.kelvin.hardy}
\left( \begin{array}{c}
\frac{\partial \tilde{x}_l }{\partial \tau}   \\[1mm]
\frac{\partial \tilde{y}_l }{\partial \tau}    \end{array}\right) = \left( \begin{array}{cc}
 -\gamma_l & 1\\
 -h(\eu^{-\tau}) & -\alpha_l
\end{array} \right)
\left( \begin{array}{c} \tilde{x}_l \\ \tilde{y}_l  \end{array}\right) +\left(
\begin{array}{c} 0 \\-
g_l(\tilde{x}_l,-\tau)\end{array}\right) .
\end{equation}

We stress that~(\ref{si.kelvin.hardy}) is obtained from~(\ref{si.hardy})
simply by changing the values of the parameters $(\alpha_l,\gamma_l)$ into
$(-\gamma_l,-\alpha_l)$, and evaluating the functions $g_l(x,t)$ and $h(\eu^t)$ in  $-\tau$ in spite of $\tau$.
We give the details of the computation for   reader's convenience.
Let us set $f_h(u,r):=\frac{h(r)}{r^2}u+f(u,r)$ and introduce $\tilde f_h$ as in~\eqref{kelvinvero}, then
\begin{equation*}
\begin{split}
 \frac{\partial}{\partial \tau}\tilde{x}_l(\tau)&=- \gamma_l \, \tilde{u}(\eu^{\tau}) \,\eu^{-\gamma_l \tau}+
\tilde{u}'(\eu^{\tau}) \eu^{(-\gamma_l+1) \tau}=-\gamma_l \, \tilde{x}_l(\tau)+\tilde{y}_l(\tau) \\
\frac{\partial}{\partial \tau}\tilde{y}_l(\tau)&=\frac{\partial}{\partial \tau}
\big[\big(\tilde{y}_l(\tau)\,\eu^{\alpha_l \tau}\big)\,\eu^{-\alpha_l \tau}\big]
=-\alpha_l \, \tilde{y}_l(\tau) +\eu^{-\alpha_l \tau}\frac{\partial}{\partial \tau}
\big[ \tilde{u}'(\eu^{\tau})\,\eu^{(n-1) \tau} \big]\\ &
=-\alpha_l \,\tilde{y}_l(\tau)- \tilde{f}_h(\tilde{u}(\eu^{\tau}),\eu^{\tau})
\eu^{(n-\alpha_l) \tau}\\ &
=-\alpha_l \,\tilde{y}_l(\tau)- f_h(\tilde{u}(\eu^{\tau})\eu^{(n-2)\tau},\eu^{-\tau})
\eu^{-(\alpha_l+2)\tau}\\ &=
-\alpha_l \, \tilde{y}_l(\tau)- f_h(\tilde{x}_l(\tau)\eu^{\alpha_l\tau},\eu^{-\tau})
\eu^{-(\alpha_l+2)\tau}\\
&=-\alpha_l \,\tilde{y}_l(\tau)-h(\eu^{-\tau}) \tilde{x}_l(\tau) -g_l(\tilde{x}_l(\tau),-\tau) \,.
\end{split}
\end{equation*}

\medbreak

Let us assume that $f(u,r)$ satisfies $\Gupiu$ with $l_u=\bar{l}_u>2_*$  (respectively
$\Gspiu$ with $l_s=\bar{l}_s>2_*$) then $\tilde{f}(u,r)$ satisfies $\Gspiu$ with $l_s=\bar{L}_s>2_*$  (resp.
$\Gupiu$ with $l_u=\bar{L}_u>2_*$), where
%\begin{equation}\label{nuovaL}
$$
 \bar{L}_s= 2- \frac{2}{\gamma_{\bar{l}_u}} =\frac{2[\bar{l}_u(n-1)-2n]}{\bar{l}_u(n-2)-2n+2} \,;
\quad \bar{L}_u= 2- \frac{2}{\gamma_{l_s}} =\frac{2[\bar{l}_s(n-1)-2n]}{\bar{l}_s(n-2)-2n+2}
$$
%\end{equation}
In particular we have $\alpha_{\bar L_s} =- \gamma_{\bar l_u}$ and $\gamma_{\bar L_s} =- \alpha_{\bar l_u}$ (resp.
$\alpha_{\bar L_u} =- \gamma_{\bar l_s}$ and $\gamma_{\bar L_u} =- \alpha_{\bar l_s}$).

We emphasize that, if $g_{l}(x,T)$ is $T-$independent, the condition $\alpha_l+\gamma_l>0$
means that~(\ref{hardy}) is subcritical (respectively $\alpha_l+\gamma_l<0$ and $\alpha_l+\gamma_l=0$
mean~(\ref{hardy}) supercritical and critical). Hence, it is clear that when we pass
from~(\ref{si.hardy}) to~(\ref{si.kelvin.hardy}), the unstable manifold $W^u(T)$ is driven
into the  stable manifold $W^s(-T)$ and viceversa, and a subcritical system is driven into a
supercritical system.

Moreover we emphasize that, in presence of a Hardy potential, we have e.g.
\begin{equation}\label{hardyinversion}
2_*(\beta) < \bar l_u < 2^* \quad\Rightarrow\quad 2^*<\bar L_s<{\rm I}(\beta)\,.
\end{equation}

\subsection{Statement of the results.}\label{secST}

\medbreak

Let us introduce some further assumptions we will assume together with
$\Gumeno$   and $\Gsmeno$.
\begin{description}
\item[$\Lu$] There is $\mathfrak{T} \in \RR$,   such that
$ \frac{g_{l_u}(x,t)}{x} \le 0$   for any $x \in \RR$ and  any $t < \mathfrak{T}$  and $\liminf\limits_{|x| \to +\infty}\frac{g_{l_s}(x,t)}{x} \ge 0$
for any $t >\mathfrak{T}$.
\item[$\Ls$] There is $\mathfrak{T} \in \RR$,   such that
$ \frac{g_{l_s}(x,t)}{x} \le 0$   for any $x \in \RR$ and  any $t > \mathfrak{T}$  and $\liminf\limits_{|x| \to +\infty}\frac{g_{l_u}(x,t)}{x} \ge 0$
for any $t <\mathfrak{T}$.
\end{description}

We  stress that $\Lu$ is trivially satisfied if   $f$ is as in~(\ref{unaeffe}a) and \textbf{K} holds, just setting
  $\mathfrak T = \ln R$. By symmetry, if $f$ is as in~(\ref{unaeffe}b) and \textbf{K} holds, $\Ls$ follows.

We are now ready to state the main results, using the terminology introduced in Definition~\ref{defRFS}.
\begin{thm}\label{tm3}
Consider~\eqref{hardy}
 and assume $\Gspiu$ with $K>0$ and  $l_s\in(2^*, {\rm I}(\beta))$,   $\Gumeno$ and $\Lu$.
Then there are four positive strictly increasing sequences $(A_k^\pm)_{k\geq0}$ and $(B_k^\pm)_{k\geq0}$ such that
$u(r,A_0^+)=v(r,B_0^+)$ is a positive \sol R0f, respectively $u(r,-A_0^-)=v(r,-B_0^-)$ is a negative \sol R0f.

For any $k>1$,
$u(r,\pm A_k^\pm)$ is a \sol Rkf. In particular we have,
$u(r,\pm A_{2j}^\pm)=v(r,\pm B_{2j}^\pm)$ and $u(r,\pm A_{2j+1}^\pm)=v(r,\mp B_{2j+1}^\mp)$.
Moreover $u(r,d)$ is a positive \sol R0s for any $0<d<A_0^+$, and
  for any $k>0$, there is $a_k^+ \in [A_{k-1}^+; A_k^+)$ such that
  $u(r,  d)$ is a \sol Rks whenever $d \in (a_{k}^+,A_{k}^+)$. An analogous
  statement holds for \sol R{}s $u(r,d)$ where $d<0$.
\end{thm}

We recall that if  $f$ satisfies $\Gspiu$ with $K>0$ and  $l_s \in (2^*, {\rm I}(\beta))$, $\Gumeno$ and $\Lu$,
then  $\tilde{f}$ obtained via~\eqref{kelvinvero} satisfies  $\Gupiu$ with  $K>0$ and   $l_u \in (2_*(\beta), 2^*)$, $\Gsmeno$ and $\Ls$.
Moreover \Rsol-solutions are turned into \fdsol-solutions and viceversa, and \Ssol-solutions into \sdsol-solutions, see Subsection~\ref{sec2.4}.
So, applying Kelvin inversion on Theorem~\ref{tm3}, we obtain the following result.

\begin{thm}\label{tm4}
Consider~\eqref{hardy}
and assume $\Gupiu$ with $K>0$ and $l_u \in (2_*(\eta), 2^*)$, $\Gsmeno$ and $\Ls$.
Then there are four positive strictly increasing sequences $(A_k^\pm)_{k\geq0}$ and $(B_k^\pm)_{k\geq0}$ such that
$u(r,A_0^+)=v(r,B_0^+)$ is a positive \sol R0f, while $u(r,-A_0^-)=v(r,-B_0^-)$ is a negative \sol R0f.

For any $k>1$,
$u(r,\pm A_k^\pm)$ is a \sol Rkf. In particular we have,
$u(r,\pm A_{2j}^\pm)=v(r,\pm B_{2j}^\pm)$ and $u(r,\pm A_{2j+1}^\pm)=v(r,\mp B_{2j+1}^\mp)$.

Moreover $v(r,L)$ is a positive \sol S0f for any $0<L<B_0^+$,
and for any $k>0$, there is $b_k^+ \in [B_{k-1}^+; B_k^+)$ such that
  $v(r, L)$ is a \sol Skf  whenever $L \in (b_{k}^+,B_{k}^+)$.
An analogous
  statement holds for \sol S{}f $v(r,L)$ where $L<0$.
\end{thm}

\section{Proofs.}\label{sec3}

\subsection{Preliminary lemmas.}\label{prellemmas}

  For every solution $\boldsymbol{x_l}(t)=(x_l(t),y_l(t))$ of~\eqref{si.hardy}, we
introduce polar coordinates
\begin{equation}\label{polcoord}
\rho_{l}= \| \boldsymbol{x_l} \| \,, \qquad \phi_{l}= \arctan(y_{l}/ x_{l})\,.
\end{equation}

\noindent
Taking into account~\eqref{transf1},
we stress that if we switch between different values of $l$, say $l_1$ and $l_2$, we get
$\rho_{l_2}=\textrm{exp}[(\alpha_{l_2}-\alpha_{l_1})t]\rho_{l_1}$ and $\phi_{l_1}(t)=\phi_{l_2}(t)$, so we can drop
the subscript in $\phi$.\\

In particular,
the next remark easily follows from the fact that the flow on
the $y$-axis rotates clockwise, i.e.,
\begin{equation}\label{flowyaxis}
\dot x_l(t)y_l(t)>0 \quad\text{ when } x_l(t)=0 \text{ and } y_l(t)\neq 0\,.
\end{equation}
Let us denote by $\Intera [x]$ the integer part of $x$.
\begin{remark}\label{ws}
Consider the trajectory of a solution $\boldsymbol{x_l}(t)$ of~\eqref{si.hardy};\\
 then
 $\Intera\left[\frac12 +\phi (t)/\pi \right]$ is decreasing in $t$.
\end{remark}

Let us denote by $\Theta^u(\tau)=\arctan(m^u(\tau))$ and by $\Theta^s(\tau)=\arctan(m^s(\tau))$, see \eqref{ellus}; we assume w.l.o.g. that
these functions are continuous,  and $\Theta^u(\tau) \to \Theta^u(-\infty)$ and $\Theta^s(\tau) \to \Theta^s(+\infty)$ as
$\tau \to -\infty$ and as $\tau \to +\infty$ respectively. Then we have the following.

\begin{lemma}\label{nongira}
Assume $\Hspiu$ with $K>0$ and $l_s \in (2^*, {\rm I}(\beta))$, then $m^s(\tau) < -\frac{n-2}{2}$ for any $\tau \in \RR$.
\end{lemma}
\begin{proof}
From Remark~\ref{differenze2} we have $m^s(+\infty)= -(n-2-\kappa(\beta)) <
-\frac{n-2}{2}$ so that the conclusion easily follows from $\Gspiu$ for $\tau$ sufficiently large, say $\tau\geq T_2$.
Moreover by $\Gspiu$ we see that for any $\tau \in \RR$ we have
$g_{l_s}(x,t) \le  A(t) \, x \,\Delta(x)$
for any $t\geq \tau$,
where $\Delta(x)$ is a continuous increasing function such that $\Delta(0)=0$, and
$A(t)$ is a continuous function such that $\lit A(t)$ is positive and finite.
Therefore there is
 $\delta(\tau)>0$ such that
\begin{equation}
\label{stimagg}
\frac{g_{l_s}(x,t)}{x}< \frac{(n-2)^2}{4}- h(\eu^{t}) \quad \text{if } t\ge \tau
 \text{ and } |x| \le \delta(\tau)\,.
\end{equation}

Let us now consider the triangle $\mathcal T(\tau)$ having vertices $\boldsymbol{O}=(0,0)$,
$\boldsymbol{A}(\tau)=(\delta(\tau),- \frac{n-2}{2}\delta(\tau))$,
$\boldsymbol{B}(\tau)=(0,- \frac{n-2}{2}\delta(\tau))$,
and denote by $o(\tau),a(\tau),b(\tau)$ the edges opposite to $\boldsymbol{O}$,
$\boldsymbol{A}(\tau)$, $\boldsymbol{B}(\tau)$, without endpoints.

If $\boldsymbol{x_{l_s}}(t_o) \in b(\tau)$, for $t_o \ge \tau$, applying~\eqref{stimagg}
we find
\begin{equation}\label{z1}
\begin{split}
&\left. \frac{d}{dt} \left(   y_{l_s} + \frac{n-2}{2} x_{l_s} \right)\right|_{t=t_o} \\
& =
\frac{n-2}{2} x_{l_s}(t_o) \left( \alpha_{l_s}- \gamma_{l_s}- \frac{n-2}{2} \right) - h(\eu^{t_o}) x_{l_s}(t_o) -
g_{l_s}(x_{l_s}(t_o),t_o)  \\
& = x_{l_s} (t_o) \left[ \frac{(n-2)^2}{4}  - h(\eu^{t_o}) -
\frac{g_{l_s}(x_{l_s}(t_o),t_o)}{x_{l_s}(t_o)} \right]>0  \,.
\end{split}
\end{equation}
Thus the flow on $b(\tau)$ points towards the exterior of $\mathcal T(\tau)$,
 whenever $t \ge \tau$. Moreover by construction
the flow of~\eqref{si.hardy}  on $a(\tau)$ points towards the exterior of~$\mathcal T(\tau)$. Finally
observe that if $\boldsymbol{x_{l_s}}(t_o) \in o(\tau)$
for $t_o \ge \tau$, we have
\begin{equation}\label{z2}
\begin{split}
\left. \frac{d}{dt}    y_{l_s} \right|_{t=t_o}
&  \ge
\left[\frac{n-2}{2} |\gamma_{l_s}| -h(\eu^{t_o}) -
\frac{g_{l_s}(x_{l_s}(t_o),t_o)}{x_{l_s}(t_o)} \right] x_{l_s}(t_o)> 0
\end{split}
\end{equation}
where we have used $|\gamma_{l_s}|>\frac{n-2}{2}$, being $l_s>2^*$.

So we can apply Lemmas~\ref{boxs} and~\ref{box}, thus finding that there  is a connected
subset $\bar W^s(\tau)\subset T(\tau) \cap W^s_{l_s}(\tau)$ containing the
origin and a point in $o(\tau)$. It follows that locally $\ell^s(\tau) \subset T(\tau)$ too, therefore $m^s(\tau)<-\frac{n-2}{2}$.
\end{proof}

Assume that we are in the hypotheses of Theorem~\ref{tm3}.
	
From $\Hspiu$  the manifolds
 $W^s_{l_s}(\tau)$ exist for any $\tau \in (-\infty, +\infty]$. Moreover from
  $\Humeno$ we see that $W^u_{l_u}(\tau)$  exists for any  $t\in \RR$ and its tangent at the origin is $y=-m^u(\tau) x$
   where  $m^u(\tau) \to m^u(-\infty)=\arctan(-\kappa(\eta))$  as $\tau \to -\infty$, see Remark~\ref{differenze2}.

\medbreak

Let $\Sigma^{u,\pm}_{l_u}(\cdot,\tau) :  [0,+\infty) \to W^{u,\pm}_{l_u}(\tau)$ and
$\Sigma^{s, \pm}_{l_s}(\cdot, \tau): [0, +\infty)\to W^{s,\pm}_{l_s}(\tau)$
 be   smooth parameterizations respectively of $W^{u,\pm}_{l_u}(\tau)$ and
$W^{s,\pm}_{l_s}(\tau)$ such that $\Sigma^{u,\pm}_{l_u}(0,\tau)=(0,0)$ and $\Sigma^{s,\pm}_{l_s}(0, \tau)=(0,0)$.
We assume w.l.o.g. that the functions $\Sigma^{u,\pm}_{l_u}: [0, +\infty) \times (-\infty,+\infty) \to \RR^2$
and $\Sigma^{s,\pm}_{l_s}: [0, +\infty) \times (-\infty,+\infty) \to \RR^2$
are continuous in both the variables.

The  following result establishes a correspondence between trajectories of
 ~(\ref{si.hardy}) and solutions of~(\ref{hardy}).
\begin{lemma}\label{paramreg}
Assume the hypotheses of Theorem~\ref{tm3} and fix $T \in \RR$. Let $u(r,d(\omega))$ be the
 \Rsol-solution of~\eqref{hardy} corresponding to $\boldsymbol{x_{l_u}}(t,T,\Sigma^{u,+}_{l_u}(\omega,T))$.\\
Then, $d(\omega) \ge 0$ is a strictly increasing function such that $d(0)=0$.
Moreover, let $v(r,L(\sigma))$ be the \fdsol-solution of~\eqref{hardy} corresponding to $\boldsymbol{x_{l_s}}(t,T,\Sigma^{s,+}_{l_s}(\sigma,T))$. Then, $L(\sigma)$ is a strictly increasing function such that $L(0)=0$.
\end{lemma}
The proof is postponed to Appendix~\ref{app2}, however  see   ~\cite[Lemma 2.10]{DF}
  for the simpler case where  $h(r) \equiv 0$ and \textbf{C} is assumed.

Lemma \ref{paramreg} permits us to introduce an additional parametrization on our manifolds depending on the parameters $d$ and $L$.

\begin{remark}\label{Ups}
For every $\tau\in \RR$, we can parametrize $W^{u,+}_{l_u}(\tau)$ directly with $d$ and $W^{s,+}_{l_s}(\tau)$ with $L$.
   In this way we can introduce the new parametrizations
$\Upsilon^{u,+}_{l_u}(\cdot,\tau): [0, d_{\tau}^+) \to W^{u,+}_{l_u}(\tau)$
and $\Upsilon^{s,+}_{l_s}(\cdot,\tau): [0, L_\tau^+) \to W^{s,+}_{l_s}(\tau)$,
which are continuous in both the variables; here  $d_\tau^+,L_\tau^+ \in (0,+\infty]$  have been defined in~\eqref{eqrhod}. However in this case $\Upsilon^{u,+}_{l_u}$ cannot be extended
continuously to $\tau=-\infty$ and $\Upsilon^{s,+}_{l_s}$ cannot be extended to $\tau=+\infty$, since
 $\Upsilon^{u,+}_{l_u}(d,\tau) \to (0,0)$ as $\tau \to -\infty$ and $\Upsilon^{s,+}_{l_s}(L,\tau) \to (0,0)$ as $\tau \to +\infty$, for any $d \in (d_{\tau}^-,d_{\tau}^+)$
 and $L \in (L_{\tau}^-,L_{\tau}^+)$.

We underline that, in general, we do not have $d_\tau^+=+\infty$ or $L_\tau^+=+\infty$. E.g., by Remark~\ref{continuabilita}, once fixed $\tau\in\RR$, if $\varrho_d<\eu^\tau$ for a certain $d>0$ then $d_{\tau}^+\leq d$.
An analogous statement holds for  $W^{u,-}_{l_u}(\tau)$ and $W^{s,-}_{l_s}(\tau)$.
\end{remark}
\begin{remark}\label{prol0}
Assumption $\Lu$ guarantees both forward and backward continuability of the trajectories of
\eqref{si.hardy} for $t \ge \mathfrak T$. Therefore $d_{\tau}^+=d^+_{ \mathfrak T}$, and $L^{\pm}_{\tau}=\pm \infty$
for any $\tau \ge  \mathfrak T$.
\end{remark}

\medbreak

Using~(\ref{cambioL})  we can define for $W^{s,\pm}_{l_u}(\tau)$ the continuous parameterizations
 $\Sigma^{s,\pm}_{l_u}(\omega,\tau)=\Sigma^{s,\pm}_{l_s}(\omega,\tau) \eu^{(\alpha_{l_u}-\alpha_{l_s}) \tau}$. It is straightforward to check that the property explained in Lemma~\ref{paramreg} holds for $\Sigma^{s,\pm}_{l_u}(\omega,\tau)$, too.

 We need to consider also the following parametrizations in polar coordinates of the manifolds:
\begin{equation}\label{sigmaparam}
\begin{array}{ll}
\Sigma^{u,\pm}_{l}(\omega, \tau)& \!\!\!\!=  \rho^{u,\pm}(\omega, \tau;l) \big( \cos( \theta^{u,\pm}(\omega, \tau)), \sin( \theta^{u,\pm}(\omega, \tau))\big)\,,\\
\Sigma^{s,\pm}_{l}(\sigma, \tau)& \!\!\!\!=  \rho^{s,\pm}(\sigma, \tau;l) \big( \cos( \theta^{s,\pm}(\sigma, \tau)), \sin( \theta^{s,\pm}(\sigma, \tau))\big) \,;
\end{array}
\end{equation}
 \begin{equation}\label{upsilonparam}
\begin{array}{ll}
\Upsilon^{u,\pm}_{l}(d, \tau)& \!\!\!\!=  R^{u,\pm}(d, \tau;l) \big( \cos( \Theta^{u,\pm}(d, \tau)), \sin( \Theta^{u,\pm}(d, \tau))\big)\,,\\
\Upsilon^{s,\pm}_{l}(L, \tau)& \!\!\!\!=  R^{s,\pm}(L, \tau;l) \big( \cos( \Theta^{s,\pm}(L, \tau)), \sin( \Theta^{s,\pm}(L, \tau))\big) \,.
\end{array}
\end{equation}
We recall that the angular coordinate of the parametrizations does not depend on the choice of $l$ as stated in Lemma~\ref{corrispondenze}.

Let $\Q\in W^{s,+}_{l}(\tau)$
  and consider the trajectory $\boldsymbol{x_{l}}(t,\tau,\Q)$
   of~\eqref{si.hardy}. Note that
  $\frac{\boldsymbol{x_{l}}(t,\tau,\Q)}{|\boldsymbol{x_{l}}(t,\tau,\Q)|}$ approaches
  $\ell^s(+\infty)$ as $t \to +\infty$,
   see~\eqref{ellus}.
  Using the polar coordinates introduced in~\eqref{polcoord}, let us set
\begin{equation}\label{traj}
\begin{array}{cc}
 \boldsymbol{x_{l}}(t,\tau,\Q)= \rho_{l}(t,\tau,\Q) \big(\cos(\phi(t,\tau,\Q),
 \sin(\phi(t,\tau,\Q) \big) \,,
\end{array}
 \end{equation}
 where we can assume that the angular coordinate $\phi$
   satisfies
   $\lit \phi(t,\tau,\Q)=-\arctan (n-2-\kappa(\beta))$.
   Similarly, if we consider $\R \in W^{u,+}_{l}(\tau)$ with the trajectory $\boldsymbol{x_{l}}(t,\tau,\R)$, then
   $\frac{\boldsymbol{x_{l}}(t,\tau,\R)}{|\boldsymbol{x_{l}}(t,\tau,\R)|}$
  approaches $\ell^u(-\infty)$ as $t \to -\infty$, and we can assume $\lito \phi(t,\tau,\R)=-\arctan(\kappa(\eta))$.
   However, if $\Q \in W^{u,+}_{l}(\tau) \cap W^{s,+}_{l}(\tau)$ for a certain $l$, we must choose one of the two conditions, the other will be satisfied up to a multiple of $2\pi$.

\subsection{Proof of the main theorems.}\label{secprooftm3}

In this section we provide the proof of Theorem~\ref{tm3}. The proof is based on some geometrical observations on the phase portrait. Then Theorem~\ref{tm4} follows from Kelvin inversion.

    We recall that the manifolds $W^{s,+}_{l_s}(\tau)$ are sets of initial conditions converging to the origin and
 a priori they are not graphs of solutions unless the system is autonomous. However, for system~\eqref{si.hardy} the number of rotations
 around the origin performed
  by $W^{s,+}_{l_s}(\tau)$
  from the origin until a point $\Q \in W^{s,+}_{l_s}(\tau)$, equals the number of rotations
 performed by the trajectory $\boldsymbol{x_{l_s}}(t,\tau,\Q)$
  for $t \ge \tau$, with reversed sign.

  More precisely
  we have the following property, the proof is adapted from  ~\cite[Propositions 3.5, 3.8]{DF}
  (we refer also to~\cite{BDF,Fdie,JK} for more details) and it is postponed to Appendix~\ref{app2}.

\begin{lemma}\label{thetaphiS}
 Let us consider system~\eqref{si.hardy} and assume  $\Hspiu$ with $K>0$ and $l_s \in (2^*,{\rm  I}(\beta))$.
   Consider the trajectory in~\eqref{traj} with $\Q = \Sigma^{s,+}_{l_s}(\sigma, \tau)$, using the notation    in~\eqref{sigmaparam}. Then
 if $h$ is a constant,  the angle   $\theta:=\theta^{s,+}(\sigma,\tau)-\theta^{s,+}(0,\tau)$ performed by
the  stable manifold $W_{l_s}^{s,+}(\tau)$
equals, but with reversed sign, the angle $\phi:=\phi(+\infty,\tau,\Q)-\phi(\tau,\tau,\Q)$
performed by the trajectory $\boldsymbol{x_{l_s}}(t,\tau,\Q)$. If
 $h$ is a function, the difference is
 \begin{equation}\label{angle}
 \begin{aligned}
| \theta-(-\phi)| &= |-\theta^{s,+}(0,\tau)+ \theta^{s,+}(0,+\infty) |\\
& = | \arctan(m^s(\tau)) -\arctan(n-2-\kappa(\beta)) | \le  \pi \, .
 \end{aligned}
 \end{equation}
\end{lemma}

We wish to underline that a similar statement can be obtained for the unstable manifold too.
Moreover, notice that Lemma~\ref{thetaphiS} is
    independent from the parametrization
 of the stable manifold, so we can use the one defined in~\eqref{upsilonparam}.

\begin{lemma}\label{spiral}
Assume $\Hspiu$ with $K>0$, $l_s \in (2^*, {\rm I}(\beta))$ and $\Lu$, then   $W^{s,+}_{l_s}(\tau)$ and $W^{s,-}_{l_s}(\tau)$ are   spirals
 rotating indefinitely  counterclockwise around the origin
for any $\tau \ge \mathfrak{T}$;
therefore $\lim_{\sigma \to +\infty} \theta^{s,+}(\sigma,\tau)=+\infty$, and
$\lim_{\sigma \to +\infty} \theta^{s,-}(\sigma,\tau)=+\infty$ for any $\tau \ge  \mathfrak{T}$.
\end{lemma}
\begin{proof}
We prove the Lemma just for $W^{s,+}_{l_s}(\tau)$; the case of $W^{s,-}_{l_s}(\tau)$ can be obtained analogously.
The Lemma is known if the system is autonomous, so we have it trivially for $\tau=+\infty$.
In fact the manifold
$W^{s,+}_{l_s}(+\infty)$ coincides with the stable manifold $M^{s,+}$ of the autonomous system~(\ref{si.hardy})
where $g_{l_s}(x,t) \equiv K g_{l_s}^{+\infty}(x)$. We recall that from $\Lu$ we get continuability of the solutions for any $t \ge \mathfrak{T}$.
 From Remark~\ref{allinfinito}, we see that for any integer $M>0$ there is $T>0$ such that  $W^{s,+}_{l_s}(\tau)$ crosses transversally
the coordinate axes, and performs at least $M$ complete rotations
  counterclockwise, for any $\tau \geq T$. In particular, using~\eqref{upsilonparam}, there exists $L_{M}$ such that $\Theta^{s,+}(L_{M},T)=2\pi {M}+\pi/2$.
Call $\boldsymbol{Q_{M}}=\Upsilon^{s,+}_{l_s}(L_{M},T)$ which belongs to the positive $y$-semiaxis. In particular $\Upsilon^{s,+}_{l_s}([0,L_{\bar{M}}]\times\{T\})$ performs more than $M$
   complete rotations.

 Let $\phi(t,T,\boldsymbol{Q_{{M}}})$ be the angular coordinate of $\boldsymbol{x_{l_s}}(t,T,\boldsymbol{Q_{{M}}})$, see~\eqref{traj}, then from Lemma
\ref{thetaphiS} we see that $\phi(T,T,\boldsymbol{Q_{ M}})=\Theta^{s,+}(L_{ M},T)=2\pi {M}+\pi/2$.
By~\eqref{flowyaxis} and Remark~\ref{ws}, we have,  $\phi(\tau,T,\boldsymbol{Q_{\bar M}}) \ge 2\pi M+\pi/2$
for any $\tau \le T$, as long as $\phi(\tau,T,\boldsymbol{Q_{\bar M}})$ exists, therefore at least for $\tau \in [\mathfrak{T}, T]$. Then using again Lemma
\ref{thetaphiS}  and Lemma~\ref{nongira}, we see that
$$
\Theta^{s,+}(L_{M},\tau)- \Theta^{s,+}(0,\tau)\geq 2\pi M+\frac{\pi}{2}-\arctan(m^s(\tau)) \ge
\left(2+\frac{1}{2} \right)\pi M \,,
$$
 thus obtaining that $\Upsilon^{s,+}_{l_s}([0,L_{M}]\times\{\tau\})$ draws more than $M$ rotations for every~$\tau \ge \mathfrak{T}$. We have thus proved the Lemma, since $M$ is arbitrarily large.
\end{proof}

\begin{lemma}\label{wu.hardy}
Assume $\Humeno$ and $\Lu$;
then
 $\limsup_{\omega \to +\infty} \rho^{u,\pm}(\omega,\tau;l_u) =+\infty$, and
$$\ds  \begin{array}{ccccc}
-\arctan\left(\frac{n-2}{2} \right)&<& \theta^{u,+}(\omega,\tau) &<&\pi/2 \\[2mm]
-\pi-\arctan\left(\frac{n-2}{2} \right)&<& \theta^{u,-}(\omega,\tau) &<& -\pi/2
\end{array} $$
 for every $\omega>0$ and for any $\tau \le \mathfrak{T}$.
\end{lemma}

\noindent

\begin{proof}
As usual, we give the proof for $W^{u,+}_{l_u}(\tau)$, the other follows similarly.
Let
$$
S(\xi) = \sup \left( \left\{  -h(\eu^t) - \frac{g_{l_u}(x,t)}{x} \,:\, 0<x\leq \xi\,, t\leq \mathfrak{T} \right\}  \cup \{0\} \right)\, .
$$
Notice that  $S(\xi)$ is  increasing, and by  $\Humeno$,
$S(\xi) <+\infty$ for any fixed $\xi$. Let $\mathfrak{m}(\xi)= \sqrt{|S(\xi)|}$,
and set $\boldsymbol{A}(\xi)=(\xi,\mathfrak{m}(\xi) \xi)$,
$\boldsymbol{B}(\xi)=(\xi,-\frac{n-2}{2}\xi)$.

This proof is analogous to the one of Lemma \ref{nongira} and relies on Lemma \ref{box}.
We construct
the triangle $Z(\xi)$ with vertices $\boldsymbol{O}$, $\boldsymbol{A}(\xi)$,
$\boldsymbol{B}(\xi)$, and  with  edges $o(\xi)$, $a(\xi)$, $b(\xi)$
opposite to the vertices $\boldsymbol{O}$, $\boldsymbol{A}(\xi)$,
$\boldsymbol{B}(\xi)$ respectively.

Let  $\boldsymbol{x_{l_u}}(t_o) \in  b(\xi)$ at a certain time $t_o\leq
\mathfrak T$,
then
\begin{equation}\label{z4}
\begin{split}
& \left. \frac{d}{dt} \big(  y_{l_u} - \mathfrak m(\xi) x_{l_u} \big) \right|_{t=t_o} \\
&= - \left( \mathfrak m(\xi) (\mathfrak m(\xi) + n -2)
+h(\eu^{t_o}) + \frac{g_{l_u}(x_{l_u}(t_o),t_o)}{x_{l_u}(t_o)}
\right) x_{l_u}(t_o) \\
&< -\left( S(\xi)+h(\eu^{t_o}) + \frac{g_{l_u}(x_{l_u}(t_o),t_o)}{x_{l_u}(t_o)}
\right) x_{l_u}(t_o) \le 0 \,.
\end{split}
\end{equation}
  From  $\Lu$ we see that if
 $\boldsymbol{x_{l_u}}(t_o) \in a(\xi)$ at a certain time $t_o\leq \mathfrak T$, then
\begin{equation}\label{z3}
\begin{split}
& \left. \frac{d}{dt} \left(   y_{l_u} + \frac{n-2}{2} x_{l_u} \right)\right|_{t=t_o}\\
& =
\frac{n-2}{2} x_{l_u}(t_o) \left( \alpha_{l_u}- \gamma_{l_u}- \frac{n-2}{2} \right) - h(\eu^{t_o}) x_{l_u}(t_o) -
g_{l_u}(x_{l_u}(t_o),t_o)  \\
& =  x_{l_u}(t_o) \left[ \frac{(n-2)^2}{4}  - h(\eu^{t_o}) -
\frac{g_{l_u}(x_{l_u}(t_o),t_o)}{x_{l_u}(t_o)} \right]>0  \,.
\end{split}
\end{equation}
So from~\eqref{z4},~\eqref{z3} the flow of~\eqref{si.hardy} on $a(\xi) \cup b(\xi)$,
 points towards the interior of $Z(\xi)$, for any $\tau \le \mathfrak T$.

Assume first $l_u > 2^*$ so that
  on $o(\xi)$ we have $\dot{x}_{l_u}>0$.
 The flow on $o(\xi)$ points towards the exterior of $Z(\xi)$;
hence we can apply Lemma~\ref{box}, and we find
 a connected subset $\mathcal{W}(\tau)\subset
( W^u_{l_u}(\tau) \cap Z(\xi) )$ containing  $\boldsymbol{O}$
 and a point in $o(\xi)$, for any $\xi>0$ and any $\tau \le \mathfrak T$.
Such a procedure can be repeated for $\xi$ arbitrarily large, thus concluding the proof of the lemma.

Now assume $l_u \in (2_*(\eta), 2^*]$. From Remark \ref{cambialu} we see that $\Gumeno$ holds for $L_u>2^*$ too,
so we can construct $W^{u,+}_{L_u}(\tau)$ and $W^{u,-}_{L_u}(\tau)$. Fix as above $t_o \le \mathfrak{T}$
 and   consider the flow of \eqref{si.hardy} for $t \le t_o$.
 We construct again the triangle $Z(\xi)$ and we observe that \eqref{z4} and \eqref{z3} continue  to hold, and the flow of
 \eqref{si.hardy} on $o(\xi)$ points outwards. So we conclude via Lemma~\ref{box} as above the existence of a subset $\mathcal{W}(\tau)$
 of
 $W^{u,+}_{L_u}(\tau)$ such that $\boldsymbol{O} \in \mathcal{W}(\tau)$ and $\mathcal{W}(\tau) \cap o(\xi) \ne \emptyset$.
 So we prove the Lemma for $W^{u,+}_{L_u}(\tau)$ for the arbitrariness of $\xi$.
 Then, recalling that  $W^{u,+}_{L_u}(\tau)$ and  $W^{u,+}_{l_u}(\tau)$ are omothetic we conclude.
\end{proof}

\medbreak

The following lemma investigates the presence of intersections between the unstable manifold $W^{u,+}_{l_u}$ and
 the manifolds $W^{s,\pm}_{l_u}$ which are omothetic to the stable manifold $W^{s,\pm}_{l_s}$.
 A similar reasoning was adopted already in~\cite{BDF,DF,FmSa,JK}.

\begin{lemma}\label{intersezioni}
Assume the hypotheses of Theorem~\ref{tm3}.
 Then
$W^{u,+}_{l_u}(\mathfrak T)$ intersects
$W^s_{l_u}(\mathfrak T)$ in a sequence of points $\boldsymbol{Q^{*,+}_j}$, for
 any $j \in \mathbb{N} $, where $\mathfrak T$ is defined in $\Lu$. Moreover, we can assume that
$\boldsymbol{Q^{*,+}_j}\in W^{s,+}_{l_u}(\mathfrak T)$ if $j$ is even, while
$\boldsymbol{Q^{*,+}_j}\in W^{s,-}_{l_u}(\mathfrak T)$ if $j$ is odd.
\end{lemma}

\begin{proof}
 Fix $\tau= \mathfrak T$.
From Lemma~\ref{spiral} we know that $W^{s,+}_{l_u}(\tau)$ and $W^{s,-}_{l_u}(\tau)$ are two spirals rotating counterclockwise around the origin, and each of them
   is cut infinitely many times by $W^{u,+}_{l_u}(\tau)$ (see Figure~\ref{sketch}):
   at least once at each rotation respectively at the point $\boldsymbol{Q_{2j}}$
   and $\boldsymbol{Q_{2j+1}}$, by the property shown in Lemma~\ref{wu.hardy}.

We develop this argument in polar coordinates too, for clarity and for later purposes.
 Using also Lemma~\ref{nongira} we see that
$$
\begin{array}{l}
 \Theta^{s,+}(0, \tau)\in (-\pi/2, 0)\,, \quad
 \Theta^{s,-}(0, \tau)\in (-3\pi/2, -\pi) \\
 \quad \big( \text{where } \Theta^{s,\pm}(0, \tau):=\lim_{L \to 0} \Theta^{s,\pm}(L, \tau)\big)\,, \\
\lim_{L \to L^{\pm}_{\tau}} \Theta^{s,\pm}(L, \tau)=+\infty\,, \\
\Theta^{u,+}(d, \tau) \in (-\pi/2, \pi/2)\,, \quad
\Theta^{u,-}(d, \tau) \in (-3\pi/2, -\pi/2)\,, \\
\lim_{d \to d_{\tau}} R ^{u,+}(d,\tau ;l_u)=+\infty\,.
\end{array}
$$
 Consider the following curves in the stripe $(\Theta,R) \in \mathcal S= \RR \times [0,+\infty)$:

\begin{equation}\label{oggi}
\begin{array}{rcl}
 \Gamma^{u,\pm}(d,\tau)&:=&(\Theta ^{u,\pm}(d,\tau), R ^{u,\pm}(d,\tau;l_u))\,, \\
%   \, , \; \textrm{ and } \; \;
  \Gamma^{s,\pm}(L,\tau)&:=&(\Theta^{s,\pm}(L,\tau), R ^{s,\pm}(L,\tau;l_u))\, ,\\
\Gamma^{s}_{2k}(L,\tau)&:=&(\Theta^{s,+}(L,\tau)- 2k \pi, R ^{s,+}(L,\tau;l_u))\,, \\
\Gamma^{s}_{2k+1}(L,\tau)&:=&(\Theta^{s,-}(L,\tau)-2 k \pi, R ^{s,-}(L,\tau;l_u))\,,
\end{array}
\end{equation}
for $k \in \mathbb{N}$.
Let us introduce, for every $k\geq 0$,
\begin{equation}\label{Larrow}
\begin{array}{rcl}
\hat L^\uparrow_{2k} &:=& \min \{ L>0 \mid (\Theta^{s,+}(L,\tau)- 2k \pi = \pi/2 \}\,,\\
\hat L^\downarrow_{2k} &:=& \min \{ L>0 \mid (\Theta^{s,+}(L,\tau)- 2k \pi = -\pi/2 \}\,,\\
\hat L^\uparrow_{2k+1} &:=& \max \{ L<0 \mid (\Theta^{s,-}(L,\tau)- 2k \pi = \pi/2 \}\,, \\
\hat L^\downarrow_{2k+1} &:=& \max \{ L<0 \mid (\Theta^{s,-}(L,\tau)- 2k \pi = -\pi/2 \}\,,
\end{array}
\end{equation}
 (except $\hat L_0^\downarrow:=0$)
and notice that $|\hat L_j^\downarrow| < |\hat L_j^\uparrow| < |\hat L_{j+2}^\downarrow|$ holds by construction, all having the same sign. In fact, roughly speaking, we have denoted with ``$\,\uparrow\,$'' the intersections of the stable manifold (the first at any lap) with the positive $y$-semiaxis and with ``$\,\downarrow\,$'' the ones with the negative; moreover even subscripts correspond to the intersections of $W^{s,+}_{l_u}(\tau)$, respectively odd subscripts to the ones of  $W^{s,-}_{l_u}(\tau)$ (cf. Figure~\ref{sketch}a).

  Let $\hat F_j^\uparrow(\tau)$ be the open region delimited by $\{ \Omega=\Gamma^s_{j}(L,\tau) \mid 0\leq |L| \leq |\hat L^\uparrow_j| \}$ and the lines $R=0$ and $\Theta=\pi/2$ (see again Figure~\ref{sketch}a).
  Similarly let $\hat F_j^\downarrow(\tau)$ be the open region delimited by $\{ \Omega=\Gamma^s_{j}(L,\tau) \mid 0\leq |L| \leq |\hat L^\downarrow_j| \}$ and the lines $R=0$ and $\Theta=-\pi/2$.
It is easy to check (see e.g. the proof of Lemma 3.9 in~\cite{DF} for details) that, for any $j \in \mathbb{N}$, the curve $d \to \Gamma^{u,+}(\cdot,\tau)$ is in $\hat F_j^\uparrow(\tau)$ for $d$ small and outside for $d$ large. So it  intersects
 the graph of
 $\Gamma^{s}_j$:   these intersections corresponds to distinct points $\bs{Q_j} \in \RR^2$.
  We are interested in the first intersections, in the sense of the parameter $d$, so let us set
\begin{equation}\label{defQed}
    \begin{aligned}
      d^*_0  := & \min\{ d > 0  \mid \textrm{$\exists L \in (0, L^+_{\tau})$ : } \;  \Gamma^{u,+}(d,\tau)=\Gamma^{s}_{0}(L,\tau) \} \,, \\
      d^*_1  := & \min\{ d > d^*_0 \mid \textrm{$\exists L \in ( L^-_{\tau},0)$ : } \;  \Gamma^{u,+}(d,\tau)=\Gamma^{s}_{1}(L,\tau) \} \,, \\
      d^*_{2k}  := & \min\{ d > d^*_{2k-1} \mid \textrm{$\exists L \in (0, L^+_{\tau})$ : } \;  \Gamma^{u,+}(d,\tau)=\Gamma^{s}_{2k}(L,\tau) \} \,,\\
      d^*_{2k+1}  := & \min\{ d > d^*_{2k} \mid  \textrm{$\exists L \in (L^-_{\tau},0)$ : } \;  \Gamma^{u,+}(d,\tau)=\Gamma^{s}_{2k+1}(L,\tau)\} \,,
    \end{aligned}
\end{equation}
for any $k \ge 1$. We denote by $L^*_j$ the unique value in $(L^-_{\tau},L^+_{\tau})$ such that
$\Gamma^{u,+}(d^*_j,\tau)=\Gamma^{s}_{j}(L^*_j,\tau)$, and we set
 $\boldsymbol{Q^{*,+}_j}(\tau):=\Upsilon^{u,+}_{l_u}(d^*_j,\tau)$,
and   $\Omega^{*}_j(\tau):= \Gamma^{u,+}(d^*_j,\tau)$, the corresponding {\em switched} polar coordinates (see Figure~\ref{sketch}a).
Moreover, notice that by construction $|\hat L_j^\downarrow|<|L_j^*| <|\hat L_j^\uparrow|$ holds, all having the same sign.
\end{proof}

\begin{figure}[t]
\centerline{\epsfig{file=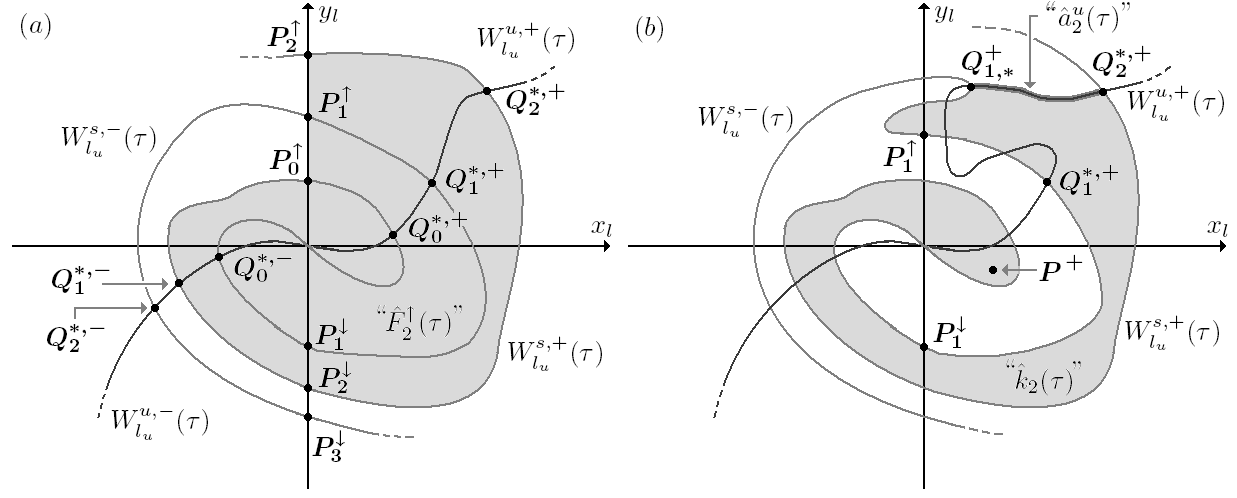, width = 12 cm}}
\caption{
We give here a sketch of the introduced notations. For simplicity we represent the sets not on the stripe $\mathcal S$ but in the $(x_l,y_l)$-plane. The points $\bs{P_i^\uparrow}$ and $\bs{P_i^\downarrow}$ correspond to the coordinate values $\hat L_i^\uparrow$ and $\hat L_i^\downarrow$  introduced in~\eqref{Larrow}.
In (a) we have colored the region ``$\hat F_2^\uparrow(\tau)$'' corresponding to $\hat F_2^\uparrow(\tau)$ for illustrative purpose:
being $W^{u,+}_{l_u}(\tau)$ unbounded, then it exits from ``$\hat F_2^\uparrow(\tau)$'' for the first time at $\bs{Q_2^{*,+}}=\Upsilon^{u,+}_{l_u}(d_2^*,\tau)=\Upsilon^{s,+}_{l_u}(L_2^*,\tau)$. So, we can verify that $|\hat L_2^\downarrow|<|L_2^*| <|\hat L_2^\uparrow|$ holds. Similarly $W^{u,-}_{l_u}(\tau)$ exits from ``$\hat F_2^\uparrow(\tau)$'' for the first time at
$\bs{Q_1^{*,-}}=\Upsilon^{u,-}_{l_u}( d_1^{*,-},\tau)=\Upsilon^{s,+}_{l_u}(L_1^{*,-},\tau)$
 and
$|\hat L^\uparrow_{0}| < |L^{*,-}_1| < |\hat L^\downarrow_{2}|$ is verified. In (b) the more tricky situation described in Lemma~\ref{nontutti} is drawn: $W^{u,+}_{l_u}(\tau)$ exits from ``$\hat F_1^\uparrow(\tau)$'' (not colored) for the first time at $\bs{Q_1^{*,+}}=\Upsilon^{u,+}_{l_u}(d_1^*,\tau)$, then it exits from
``$\hat F_2^\uparrow(\tau)$'' for the first time at $\bs{Q_2^{*,+}}=\Upsilon^{u,+}_{l_u}(d_2^*,\tau)$.
Between these points we locate $\bs{Q_{1,*}^+}=\Upsilon^{u,+}_{l_u}(d_{*,1},\tau)$ (so that $d_{*,1}>d^*_1$).
We have highlighted the set corresponding to $\hat a_2^u(\tau)$ defined in~\eqref{defauj} and colored the region corresponding to $\hat k_2(\tau)$.
We emphasize that in the simpler case $d_1^*=d_{*,1}$, $\hat a_2^u(\tau)$ consists of the whole branch between $\bs{Q_1^{*,+}}$ and $\bs{Q_2^{*,+}}$ as in (a). Finally, arguing as in the proof of Remark~\ref{regularslow} all the points of $\hat a_2^u(\tau)$ and $\hat k_2(\tau)$ tend to $\bs{P^+}$ as $t\to+\infty$.
}
\label{sketch}
\end{figure}

\begin{remark}
A similar result can be obtained for the unstable manifold $W^{u,-}_{l_u}(\tau)$ (cf. Figure~\ref{sketch}a): the curve $\Gamma^{u,-}(\cdot,\tau)$ must exit from $\hat F_j^\downarrow(\tau)$, so $W^{u,-}_{l_u}(\tau)$ intersects the stable manifold $W^{s}_{l_s}(\tau)$
in the distinct points $\boldsymbol{Q^{*,-}_j}(\tau)$  with switched polar coordinates   $\Omega^{*,-}_j(\tau)\in \mathcal S$.
In particular
$\boldsymbol{Q^{*,-}_j}(\tau)\in W^{s,-}_{l_u}(\tau)$ if $j$ is even, while
$\boldsymbol{Q^{*,-}_j}(\tau)\in W^{s,+}_{l_u}(\tau)$ if $j$ is odd. Moreover, we can identify the corresponding sequences $(d_j^{*,-})_{j\geq 0}$, $(L_j^{*,-})_{j\geq 0}$, with $d_j^{*,-}<0$ and $(-1)^j L_j^{*,-}<0$, such that
$\Omega^{*,-}_j(\tau)=\Gamma^{u,-}(d_j^{*,-},\tau)=\Gamma^s_j(L_j^{*,-},\tau)$.
 Let us set  $\hat L^\uparrow_{-1}:=0$; by construction we find $|\hat L^\uparrow_{j-1}| < |L^{*,-}_j| < |\hat L^\downarrow_{j+1}|$, all having the same sign. We stress that the sequences in the statement of Theorem~\ref{tm3} can be now introduced as
$A_j^+=d^*_j$, $A_j^-=|d^{*,-}_j|$, $B_{2k}^+=L_{2k}^*$ and $B_{2k+1}^+=L_{2k+1}^{*,-}$, $B_{2k}^-=|L_{2k}^{*,-}|$ and $B_{2k+1}^-=|L_{2k+1}^*|$: they are strictly increasing by construction.
\end{remark}

By construction, $\boldsymbol{x}_{l_u}(t;\mathfrak T, \boldsymbol{Q^{*,+}_j}(\mathfrak T))$ is a homoclinic
trajectory of~(\ref{si.hardy}), and the corresponding solution $u(r,d^*_j)$
of~(\ref{hardy}) is a \sol R{ }f solution.    We introduce the following notation:
let $\bar{\Omega}=(\bar{\Theta},\bar{R})$ be a point in $\mathcal S$, and $\boldsymbol{\bar{Q}}=
\bar{R}(\cos(\bar\Theta), \sin(\bar\Theta))$;  we denote by
$ \Omega_{l_u}(t,\mathfrak T, \bar{\Omega})=(\phi(t,\mathfrak T, \bar{\Omega}),\rho_{l_u}(t,\mathfrak T, \bar{\Omega}))$ the
switched polar coordinates of
  $\boldsymbol{x_{l_u}}(t,\mathfrak T, \boldsymbol{\bar{Q}})$, such that $ \Omega_{l_u}(\mathfrak T,\mathfrak T, \bar{\Omega})=\bar{\Omega}$ (so that it is uniquely defined).
 With a little abuse of notation we denote by $ \Omega_{l_s}(t,\mathfrak T, \bar{\Omega})=( \phi(t,\mathfrak T, \bar{\Omega}), \rho_{l_s}(t,\mathfrak T, \bar{\Omega}))$
  the switched  polar coordinates of $\boldsymbol{x_{l_s}}(t,\mathfrak T, \boldsymbol{\bar{R}})$,
  where  $\boldsymbol{\bar{R}}=\boldsymbol{\bar{Q}} \eu^{(\alpha_{l_s}-\alpha_{l_u})\mathfrak{T}}$, and we observe that the angular coordinate
  is the same as in $ \Omega_{l_u}(t,\mathfrak T, \bar{\Omega})$, while the radial coordinate is multiplied by $ \eu^{(\alpha_{l_s}-\alpha_{l_u})t}$.

\begin{lemma}\label{gsfd}
Assume the hypotheses of Theorem~\ref{tm3}.
Then, $u(r,d^*_j)$ is \sol Rjf. In particular,
 $u(r,d_0^*)$ is a positive solution.
\end{lemma}

\begin{proof}

When we consider~\eqref{eq.na} we can simply repeat the proof of~\cite[Lemma~3.11]{DF}
with no changes. When we deal with~\eqref{hardy} we need to adapt slightly  the argument: we sketch the proof
  for reader's convenience.

If $j=2k$ is even we have
 $\phi(\mathfrak T,\mathfrak T, \Omega^*_{2k})=\Theta^{u,+}(d^*_{2k},\mathfrak T)=\Theta^{s,+}(L^*_{2k},\mathfrak T)-2 k\pi$,
 while if $j=2k+1$ is odd we have
 $\phi(\mathfrak T,\mathfrak T, \Omega^*_{2k+1})=\Theta^{u,+}(d^*_{2k+1},\mathfrak T)=\Theta^{s,-}(L^*_{2k+1},\mathfrak T)-2 k\pi$, see also Lemma~\ref{thetaphiS}.
Therefore
$\boldsymbol{x_{l_u}}(t,\mathfrak T, \bs{Q^*_j})$  performs the angle $\phi(\mathfrak T,\mathfrak T, \Omega^*_{j})+\arctan(\kappa(\eta))$
around the origin when $t \in(-\infty, \mathfrak T]$, and  it performs the angle $-j \pi-\arctan(n-2-\kappa(\beta))-\phi(\mathfrak T,\mathfrak T, \Omega^*_{j})$ for every $j$
when $ t \in [\mathfrak T,+\infty)$.
Summing up, $\boldsymbol{x_{l_u}}(t,\mathfrak T, \bs{Q^*_j})$ performs the angle
$-j \pi -\arctan(n-2-\kappa(\beta))+\arctan(\kappa(\eta))$.
Since $\arctan(n-2-\kappa(\beta))-\arctan(\kappa(\eta)) \in (0, \pi)$ we see that $\boldsymbol{x_{l_u}}(t,\mathfrak T, \bs{Q^*_j})$
crosses the $y$ axis exactly $j$ times and so $u(r,d^*_j)$ changes sign exactly $j$ times.
\end{proof}

\begin{remark}\label{nontutti}
The solution $u(d^*_j,r)$ of~\eqref{hardy} is a \sol R{j}f, and it is
definitively positive for $j$ even and definitively negative for $j$ odd. However, a priori, we may find some $d>0$, $d \ne d^*_k$ for any $k\ge 0$
such that $u(d,r)$ is a \sol R{}f (cf. Figure~\ref{sketch}b). In fact it may happen e.g. that $d \to\Gamma^{u,+}(d,\mathfrak T)$ intersects
$\Gamma^{s}_{0}(\cdot,\mathfrak T)$ in, say, three points: at $d=d^*_0<d_a<d_b$, thus corresponding to three distinct points $\bs{Q^*_0},\bs{Q^a_0},\bs{Q^b_0} \in \RR^2$.
In this case we have three positive \sol R0f: $u(r,d^*_0)$, $u(r,d^a_0)$, $u(r,d^b_0)$, see~\cite[Section 3]{DF} for a more detailed discussion
of this point.
\end{remark}
\begin{remark}\label{prol}
We emphasize that   $\Gamma^{u,+}(d,\tau)$ and $\Gamma^{s}_{2k}(L,\tau)$ are well defined for any $0 \le d < d_\mathfrak T$, $0 \le L < L_\mathfrak T$,
 whenever $\tau \ge \mathfrak T$, cf $\Lu$.
Next
 if $\Gamma^{u,+}(\cdot,\mathfrak T)$ intersects $\Gamma^{s}_{j}(\cdot,\mathfrak T)$ in $\Omega^*_j$, then
 $\Gamma^{u,+}(\cdot,\tau)$ intersects $\Gamma^{s}_{j}(\cdot,\tau)$ in a point $\Omega^*_j(\tau)$
 corresponding to the same solution $u(d^*_j,r)$, for any $\tau \ge \mathfrak T$.
\end{remark}

Following the ideas of~\cite[Section 3]{DF}, let us now turn to consider the solution $u(r,d)$ where $d \ne d^*_j$ for any $j \in \mathbb{N}$.
Fix $\tau\ge \mathfrak T$, we need to define several subsets of the stripe $\mathcal S$:
\begin{equation}\label{defset1}
\begin{array}{rcl}
  \hat{A}^u_j(\tau) &:= & \{ \Omega=\Gamma^u(d,\tau) \mid 0 \le d \le d^*_j \}  \,, \\
    \hat{B}^s_{2k}(\tau) &:= & \{ \Omega=\Gamma^{s}_{2k}(L,\tau) \mid 0 \le L \le L^*_{2k} \} \,, \\
       \hat{B}^s_{2k+1}(\tau) &:= & \{ \Omega=\Gamma^{s}_{2k+1}(L,\tau) \mid L^*_{2k+1} \le L \le 0 \} \,,
\end{array}
\end{equation}
for any $j, k \in \mathbb{N}$. We denote by $\hat{E}_j(\tau)$ the open set enclosed by $\hat{A}^u_j(\tau)$,
$\hat{B}^s_{j}(\tau)$ and the line $R=0$, which is bounded for any $\tau \ge \mathfrak T$ and any $j \in \mathbb{N}$.
Further set $d_{*,-1}:=0$ and, for $j \geq 0$,
$$
d_{*,j}:= \max \{ d \in [0, d^*_{j+1}) \mid \textrm{there is $L \in \RR$ so that } \; \Gamma^{u,+}(d,\tau)=\Gamma^{s}_{j}(L,\tau) \} \,.
$$
Observe that if $\Gamma^{u,+}(\cdot,\tau)$ has a unique intersection with $\Gamma^{s}_{j}(\cdot,\tau)$
then $d_{*,j}=d^{*}_{j}$ (in general we have $d_{*,j} \ge d^{*}_{j}$). We set (cf. Figure~\ref{sketch}b)
\begin{equation}\label{defauj}
    \hat{a}^u_j(\tau):= \{ \Omega=\Gamma^u(d,\tau) \mid d_{*,j-1} < d < d^*_j \}\,.
\end{equation}
Define $\hat{k}_0(\tau) = \hat{E}_0(\tau)$ and, for any $j>0$, denote by $\hat{k}_j(\tau)$ the open bounded set enclosed by $\hat{B}^s_{j}(\tau)$, $\hat{a}^u_j(\tau)$, $\hat{B}^s_{j-1}(\tau)$
and the line $R=0$ (observe that $\hat{k}_j(\tau) \subset \hat{E}_j(\tau)$).
Note that these sets have the following property.
\begin{remark}\label{key}
If $\bar{\Omega}$ belongs to  $\hat{A}^u_j(\tau)$, $\hat{B}^s_j(\tau)$, $\hat{E}_j(\tau)$, $\hat{k}_j(\tau)$, for some $\tau \ge\mathfrak T$
then $\Omega(t,\tau,\bar{\Omega})$ belongs respectively  to $\hat{A}^u_j(t)$, $\hat{B}^s_j(t)$, $\hat{E}_j(t)$, $\hat{k}_j(t)$
for any $t \ge \mathfrak T$.
\end{remark}
\begin{proof}
We  just sketch the proof which is strongly inspired by~\cite[Section 3]{DF} and in particular~\cite[Lemma 3.14]{DF}.
The claim concerning $\hat{A}^u_j(\tau)$, $\hat{B}^s_j(\tau)$ follows from construction. Then note that if
$\bar{\Omega} \not\in\hat{A}^u_j(\tau)$ then $\Omega(t,\tau,\bar{\Omega})\not\in\hat{A}^u_j(t)$ too, and the same property holds
for $\hat{B}^s_j(\tau)$. It follows that if $\bar{\Omega} \in \hat{E}_j(\tau)$ then $\Omega(t,\tau,\bar{\Omega})$
cannot cross $\hat{A}^u_j(t)$ and $\hat{B}^s_j(t)$, hence it is forced to stay in $\hat{E}_j(t)$ for any $t \in \RR$.
The claim concerning $\hat{k}_j(\tau)$ is analogous.
\end{proof}
Now we turn to consider \sol R{}s solutions.
\begin{lemma}\label{regularslow}
Assume the hypotheses of Theorem~\ref{tm3}
and fix $\tau \ge \mathfrak T$.
Then solutions $u(r,d)$ of~\eqref{hardy} corresponding to $\Omega(t,\tau,\bar{\Omega})$ with $\bar{\Omega} \in \hat a^u_j(\tau)$ are
\sol Rjs. In particular,
 $u(r,d)$ is a positive solution, for any $0<d<d^*_0$.
\end{lemma}
\begin{proof}
From Lemma~\ref{wu.hardy}, we know that
$\hat{A}^u_j(\mathfrak T) \in \{(\Theta, R) \mid |\Theta| < \frac{\pi}{2} \}$.
Moreover $\hat{A}^u_j(\tau)$ is a path connecting
the point $\Omega_a(\tau)=(\Theta^u(\tau),0)$ to $\Omega^{*,+}_j(\tau)$,
 where $\Theta^u(\tau)=\arctan(m^u(\tau)) \in (-\arctan \frac{n-2}{2},\frac{\pi}{2})$.
From Lemma~\ref{spiral}
we see that $\hat{B}^s_j(\tau)$ is a path connecting
a point $\Omega^j_b(\tau)=(\Theta^j(\tau),0)$ to $\Omega^{*,+}_j(\tau)$, where
$\Theta^j(\tau)+j\pi =\arctan(m^s(\tau)) \to -\arctan(n-2-\kappa(\beta))$ as $\tau \to +\infty$,
and $|\Theta^j(\tau)+j\pi| <\frac{\pi}{2} $.

Let $\hat{\Omega}(t)=(\phi(t),\hat{\rho}(t))=\Omega(t,\tau,\bar{\Omega})$ be the switched polar coordinates of $\boldsymbol{x_{l_u}}(t,\tau, \Q)$ and
$\tilde{\Omega}(t)=(\phi(t),\tilde{\rho}(t))$ the switched polar coordinates of the trajectory
$\boldsymbol{x_{l_s}}(t,\tau, \Q \eu^{(\alpha_{l_s}-\alpha_{l_u})t })$ corresponding to the same solution $u(r)$ of~\eqref{hardy},
so that $\tilde{\rho}(t)=\hat{\rho}(t)\eu^{(\alpha_{l_s}-\alpha_{l_u})t }$.

Let us denote by $\tilde{A}^u_{j}(\tau):= \{ (\Theta, R \eu^{(\alpha_{l_s}-\alpha_{l_u})\tau }) \mid
(\Theta, R) \in \hat{A}^u_{j}(\tau) \}$, $\tilde{B}^s_{j}(\tau):= \{ (\Theta, R \eu^{(\alpha_{l_s}-\alpha_{l_u})\tau } \mid
(\Theta, R) \in \hat{B}^s_{j}(\tau) \}$, and similarly for $\tilde{a}^u_{j}(\tau)$, $\tilde{E}_{j}(\tau)$, $\tilde{k}_{j}(\tau)$.
By construction $\hat{\Omega}(t) \in \hat{A}^u_{j}(t), \hat{B}^s_{j}(t),
\hat{a}^u_{j}(t)$, $\hat{E}_{j}(t)$, $\hat{k}_{j}(t)$, iff $\tilde{\Omega}(t) \in \tilde{A}^u_{j}(t), \tilde{B}^s_{j}(t),
\tilde{a}^u_{j}(t)$, $\tilde{E}_{j}(t)$, $\tilde{k}_{j}(t)$.

Denote by  $\tilde{\Gamma}^{s,\pm}_{l_s}(\sigma,+\infty)$  the switched polar coordinates for $\Sigma^{s,\pm}_{l_s}(\sigma,+\infty)$, and set
$$
\begin{array}{l}
\tilde{\Gamma}^{s}_{2k,l_s}(\sigma,+\infty)=\{\tilde{\Omega}-(2k \pi ,0) \mid \tilde{\Omega} \in \tilde{\Gamma}^{s,+}_{l_s}(\sigma,+\infty) \} \, ,\\
\tilde{\Gamma}^{s}_{2k+1,l_s}(\sigma,+\infty)=\{\tilde{\Omega}-(2k \pi ,0) \mid \tilde{\Omega} \in \tilde{\Gamma}^{s,-}_{l_s}(\sigma,+\infty) \}\, .
\end{array}
$$

We define
 $\tilde{K}_{j}(+\infty)$ as the unbounded open stripe between $\tilde{\Gamma}^{s}_{j,l_s}(\sigma,+\infty)$ and $\tilde{\Gamma}^{s}_{j-1,l_s}(\sigma,+\infty)$.
 Observe further that $\tilde{K}_{j}(+\infty)$ corresponds to the unbounded open subset between
 $W^{s,+}_{l_s}(+\infty)$ and $W^{s,-}_{l_s}(+\infty)$ for~\eqref{si.hardy}, the one containing
 $\boldsymbol{P^+}$ if $j$ is even and the one the one containing
 $\boldsymbol{P^-}$ if $j$ is odd.

 Note that $\tilde{k}_{j}(\tau)$  as $\tau \to +\infty$ approaches    a bounded open subset of
 $\tilde{K}_{j}(+\infty)$, say $\tilde{k}_j(+\infty)$ containing  the switched polar coordinates either of $\boldsymbol{P^+}$ if $j$ is even
 or of $\boldsymbol{P^-}$ if $j$ is odd;
similarly $\tilde{a}^u_{j}(\tau)$
approaches
 a subset of $\tilde{K}_{j}(+\infty)$   as $\tau \to +\infty$.
  Moreover note that $\boldsymbol{P^+}$ (or $\boldsymbol{P^-}$)
  is the unique attracting subset of
 $\tilde{k}_{j}(+\infty)$.  Let $(\Theta^{\pm},R^{\pm})$ be switched  polar coordinates for $\boldsymbol{P^\pm}$.

 Then observe that
 if $\bar{\Omega} \in \hat{k}_{2k}(\tau)$ (respectively $\bar{\Omega} \in \hat{k}_{2k+1}(\tau)$)
  then $\tilde{\Omega}(\tau) \in \tilde{k}_{2k}(\tau)$ (resp. $\tilde{\Omega}(\tau) \in \tilde{k}_{2k+1}(\tau)$) and
 $\tilde{\Omega}(t)$ converges to $(\Theta^{+}-2k \pi,R^{+})$ (resp. $(\Theta^{-}-2k \pi,R^{-})$) as $t \to +\infty$.
 Moreover if $\bar{\Omega} \in \hat{a}^u_{2k}(\tau)$ (respectively $\bar{\Omega} \in \hat{a}^u_{2k+1}(\tau)$) then
 $\hat{\Omega}(t) \to (-\arctan(\kappa(\eta)),0)$ as $t \to -\infty$, and again
 $\tilde{\Omega}(\tau) \in \tilde{a}^u_{2k}(\tau)$ (resp. $\tilde{\Omega}(\tau) \in \tilde{a}^u_{2k+1}(\tau)$),
   and
 $\tilde{\Omega}(t)$ converges to $(\Theta^{+}-2k \pi,R^{+})$ (resp. $(\Theta^{-}-2k \pi,R^{-})$) as $t \to +\infty$;
 hence the corresponding solution $u(r)$ of~\eqref{hardy} is a \sol R{2k}s (resp. a \sol R{2k+1}s) which is definitively positive (resp. negative).
\end{proof}

We stress that by construction we find  $a_j^+=d_{*,j-1}$,
where $(a^+_k)_{k\geq 1}$ is the sequence in the statement of Theorem~\ref{tm3}.
Lemmas~\ref{gsfd} and~\ref{regularslow} can be reformulated also for \Rsol-solutions $u(d,r)$ with $d<0$, reasoning  similarly on the curve $\Gamma^{u,-}(\cdot,\tau)$. So the proof of Theorem~\ref{tm3} is concluded.

\section{Some further examples.}\label{MAP}

In this section we briefly present other types of nonlinearities to which our theorems apply.
Let us begin by noticing that when
$f$ is as in
 (\ref{unaeffe}a), arguing as in Section~\ref{secfow}, we can choose
$l_u$ and $l_s$ as in Corollary \ref{HC},  and we get
  $$g_{l_u}(x,t)=[K(0)+\Delta_u(-t)] |x|^{q-2}x \, , \quad g_{l_s}(x,t)= [K(\infty)+\Delta_s(t)] |x|^{q-2}x \, , $$
  where $K(0)<0<K(\infty)$, and
  $\Delta_u(T)$ and $\Delta_s(T)$ go to $0$ as $T \to +\infty$.
  In fact we can also consider logarithmic growth, e.g.
  \begin{equation}\label{loga}
   \begin{array}{cc}
    K(r)\sim K(0)|\ln(r)|^{a_0}r^{\delta_0} & \textrm{as } r \to 0 \, ,
    \end{array}
\end{equation}
where $K(0)<0$, $a_0 \in \RR$, $\delta_0>-2$. In this case assumption $\Gupiu$ does not hold
but $\Gumeno$ holds for any $l_u \in (l(q,\delta_0), {\rm I}(\eta))$, see \eqref{elleUeS}. So if \textbf{K} holds with
the first in \eqref{Kasympt} replaced by \eqref{loga} we can apply Corollary \ref{HC}.
Similarly if $f$ is as in (\ref{unaeffe}b) and \textbf{K} holds with
the second in \eqref{Kasympt} replaced by    $ K(r)\sim K(\infty)|\ln(r)|^{a_{\infty}}r^{\delta_{\infty}}$, $K(\infty)>0$,
then $\Gsmeno$ holds for any $l_s \in (2_*(\beta), l(q,\delta_{\infty}))$, so we can apply Corollary \ref{HD}.

Introduce
\begin{equation}\label{fi}
f_i(u,r)= \sum_{j=1}^{N_i} K_{i,j}(r) r^{\delta_{i,j}} |u|^{q_{i,j}-2}u
\end{equation}
for a certain integer $N_i$, where $\delta_{i,j}>-2$, $q_{i,j}>2$, $K_{i,j}$ are continuous functions which are bounded and uniformly far from zero for $r$ small and $r$ large, negative near zero, changes sign at $R_{i,j}>0$ and then they are positive. We assume
 $$
 \ell_i = \ell_{i,j} = l(q_{i,j},\delta_{i,j})\,, \quad \text{for every } j\,,
 $$
 thus having
 $$
 g_{l_i}(x,t)=\sum_{j=1}^{N_i}  K_{i,j}(\eu^t)|x|^{q_{i,j}-2}x\,.
 $$
Moreover it is possible to assume that some of the $K_{i,j}$ are identically zero for either $r\leq R_{i,j}$ or $r\geq R_{i,j}$. We do to not enter in  details for major clarity. Assume now
\begin{equation}\label{fi2}
f(u,r) = \sum_{i=1}^{m} f_i(u,r) \,, \quad \ell_1 > \ell_2 > \cdots > \ell_{m-1} > \ell_m \,.
\end{equation}
Setting $l_u=\ell_1$
and remembering that, for $L\neq l$, one has $g_L(x,t)/x = g_l(\xi,t)/\xi$, where $ \xi= x \eu^{(\alpha_l-\alpha_L)t} $, a computation gives the validity of $\Gumeno$ if $l_u<{\rm I}(\eta)$. The first condition of $\Lu$ holds simply setting $\mathfrak T = \min \{R_{i,j}\}$, while we have to assume
\begin{equation}\label{luhold}
 \text{if }  R_{i_o,j_o} = \mathfrak T  \quad \text{ then } \quad q_{i_o,j_o} \geq q_{i,j} \text{ for every } (i,j)\neq(i_o,j_o)
\end{equation}
in order to fulfill the second one. Roughly speaking, the term with the greatest power is the first to change sign.

Assumption $\Gspiu$ holds setting $l_s=\ell_m$. Hence, Theorem~\ref{tm3} applies if $f$ is as in~\eqref{fi2} with $ 2^*<\ell_m < \cdots < \ell_1<{\rm I}(\eta)$ and such that~\eqref{luhold} holds.

\medbreak

Arguing as above we can find the corresponding specular conditions in order to permit the application of Theorem~\ref{tm4}. Assume $f(u,r)= - \sum_{i=1}^{m} f_i(u,r)$ with $f_i$ as in~\eqref{fi} assuming now $\ell_1 < \ell_2 < \cdots < \ell_{m-1} < \ell_m$.
Setting $l_s=\ell_1>2_*(\beta)$ we have $\Gsmeno$, then set $\mathfrak T=\max\{R_{i,j}\}$ and assume~\eqref{luhold} so that $\Ls$ is given. The validity of $\Gupiu$ is given setting $l_u=\ell_m$, asking $2_*(\eta)<\ell_m<2^*$.

\bigbreak

The functions previously considered consist of sum of possibly different polynomial terms. However, our results permit us to consider also
more general nonlinearities, which however have a leading term in their expansion for $u$ small and $u$ large which is polynomial,
e.g.
$$
f(u,r)= \frac{|u|^{q_1-2} u}{1+u^{q_2}} \,\cdot\, \frac{r^{\delta_1} }{1+r^{\delta_2}}\,.
$$
assuming $q_1>2$, $q_1-q_2>2$ and $\delta_1>-2$, $\delta_1-\delta_2>-2$.
In such a case, a straightforward computation gives that $\Gupiu$ holds requiring $l_u=l(q_1-q_2,\delta_1)$ and
$\Gspiu$ holds setting $l_s=l(q_1,\delta_1-\delta_2)$.

Further notice that our results are robust. I.e., we have the following.
\begin{remark}\label{robust}
Assume \textbf{H} and consider (for simplicity) the functions $\tilde{f}^{+}(u,r)$,
$\tilde{f}^{-}(u,r)$  satisfying the assumption of Corollary \ref{HC}, and \ref{HD}
respectively.

Let
$\bar{f}(u,r)$ be such that $\bar{f}(0,r)=\frac{\partial \bar{f}}{\partial u}(0,r)=0$; suppose that there is $C>1$
such that $\bar{f}(u,r) \equiv 0$ for $r<1/C$ and for $r>1/C$, and that $\lim_{|u| \to +\infty} \frac{\bar{f}(u,r)}{|u|^{q-1}}=0$
uniformly for $r>0$.

Then the function $f(u,r)=\tilde{f}^+(u,r)+\bar{f}(u,r)$ satisfies the assumptions of Theorem \ref{tm3}, and the solutions
of \eqref{hardy} have the structure described in Corollary \ref{HC} (and Theorem \ref{tm3}).

Similarly   the function $f(u,r)=\tilde{f}^-(u,r)+\bar{f}(u,r)$ satisfies the assumptions of Theorem \ref{tm4}, and the solutions
of \eqref{hardy} have the structure described in Corollary \ref{HD} (and Theorem \ref{tm4}).
\end{remark}

\medbreak

The assumptions on  Hardy potentials $h(r)$ are more clear so we
  just emphasize the following interesting example satisfying \textbf{H}:
$$ \begin{array}{r}
\ds     u'' +\frac{n-1}{r}u' + \frac{C}{1+r^2} u +f(u,r)=0\,,\qquad \text{with } C<\frac{(n-2)^2}{4}\,, \\[2mm]
  \ds  \textrm{i.e.  }\, h(r)=  \frac{C r^2}{1+r^2}, \quad  \text{with } \eta=0\,, \text{ and } \beta=C \,.
   \end{array}
$$

\appendix
\section{Appendix}
\subsection{On the lack of continuability}\label{app1}

In this appendix we  first review briefly some well known facts concerning exponential dichotomy, see, e.g., \cite{Co78}.
Then we develop
   the construction
of stable and unstable manifolds  for non-autonomous systems, i.e. $W^u_{l_u}(\tau)$ and $W^s_{l_s}(\tau)$, when continuability of the trajectories of~\eqref{si.hardy} is ensured, i.e.
when hypothesis \textbf{C} holds. Then we extend our discussion to the case where
\textbf{C} does not hold.

\medbreak

Denote by $\mathcal{A}_l(t)=\left( \begin{array}{cc} \alpha_{l} &
1
\\ -h(\eu^t) & \gamma_{l}
\end{array} \right)$ the linearization of the right hand side of~\eqref{si.hardy} in the origin, and by $\mathcal{A}_l(\pm\infty)=
\lim_{t \to \pm \infty} \mathcal{A}_l(t)$.
 Assume either $\Hupiu$ or  $\Humeno$:
note that $\mathcal{A}_{l_u}(-\infty)$ has $\lambda_2<0<\lambda_1$
as eigenvalues where $\lambda_1:=\alpha_{l_u}-\kappa(\eta)$ and $\lambda_2:=\alpha_{l_u}+2-n+\kappa(\eta)$.  By \textbf{H},
 $\mathcal{A}_{l_u}(t)$ can be seen as
an $L^1$ perturbation of $\mathcal{A}_{l_u}(-\infty)$, therefore it admits exponential
dichotomy in  negative semi-lines $(-\infty,\tau]$. More precisely let $X(t)$ be the
fundamental matrix of
\begin{equation}\label{eqmatriciale}
\dot{x}=\mathcal{A}_{l_u}(t) x \, ,
\end{equation} i.e.
the matrix solution of~\eqref{eqmatriciale} such that $X(0)=I$, where $I$ denotes the identity matrix.
Then,
for any $\tau \in \RR$ there
is a  constant $K=K(\tau)>1$, exponents $\bar{\la}_2<0<\bar{\la}_1$ and a projection
$\mathcal P^-$ such that
\begin{equation}\label{dichotu}
\begin{array}{cc}
\|X(t)(I-\mathcal P^-) X(s)^{-1} \| \le K \eu^{\bar{\la}_1 (t-s)} & \textrm{ for any $t<s<\tau$} \,,\\
\|X(t)\mathcal P^- X(s)^{-1} \| \le K \eu^{\bar{\la}_2 (t-s)}  & \textrm{ for any $s<t<\tau$}\,,
\end{array}
\end{equation}
see, e.g.,~\cite[Section 4]{Co78}. Moreover the optimal choice for
$\bar{\la}_i$ is $\bar{\la}_i=\la_i$ for $i=1,2$, see~\cite[Appendix]{CF}.
Let us denote by $\mathcal  P^-(\tau):=X(\tau)\mathcal P^- X(\tau)^{-1}$, and by
 $\ell^u(\tau)$ the $1$-dimensional  kernel of $\mathcal P^-(\tau)$;
 then $\ell^u(\tau)$ is the unstable space for~\eqref{eqmatriciale}. I.e
 let  $\vec{\xi} \in \RR^2$, and denote by $\vec{\xi}(t)$ the solution of~\eqref{eqmatriciale}
 such that $\vec{\xi}(\tau)=\vec{\xi}$; then
$\vec{\xi}(t)$ is bounded for $t \le 0$ iff
 $\vec{\xi} \in \ell^u(\tau)$, cf.~\cite[Section 4]{Co78}. Since  $\ell^u(\tau)$ is $1$-dimensional
we see that there is $c=c(\vec{\xi})$ such that $\|\vec{\xi}(t)\| \eu^{-\la_1 t} \to c$ as $t \to -\infty$.
Also note that by construction $\ell^u(\tau)$ is a line, and $\vec{\xi} \in \ell^u(\tau)$ iff
$\vec{\xi}(t) \in \ell^u(t)$.

Now assume $\Humeno$ and consider~\eqref{si.hardy} where $l=l_u$: we consider this problem as
a nonlinear perturbation of~\eqref{eqmatriciale}. Thus, setting
$Q(\delta)=\{(x,y) \mid |x|\le \delta , \; |y| \le \delta \}$,
we get the following, see~\cite[Theorem 2.25]{Jsell}.

\begin{lemma}\label{loc}
Assume $\Humeno$; then for any $N \in \RR$ we can find $\delta=\delta(N)$ such that the set
\begin{equation}\label{Wubarloc}
\begin{array}{r}
   W^u_{l_u,loc}(\tau):= \Big\{ \Q \in Q(\delta) \mid  \ds\boldsymbol{x_{l_u}}(t,\tau;\Q) \in Q(\delta)  \; \textrm{for any $t \le \tau$,} \\
    \text{ and } \ds \; \lito \boldsymbol{x_{l_u}}(t,\tau;\Q) = (0,0) \Big\}
\end{array}
\end{equation}
is a graph on  $\ell^u(\tau) \cap Q(\delta)$ for any $\tau \le N$. Moreover  $\ell^u(\tau)$
is the tangent space to $W^u_{l_u,loc}(\tau)$ in the origin.
\end{lemma}

We sketch the proof for completeness.
Assume $\Humeno$ and suppose first that,
$$|g_{l_u}(x_2,t)-g_{l_u}(x_1,t)| \le c(\tau) |x_2-x_1|  \, , \qquad \textrm{for any $t \le \tau$}$$
for some $c(\tau)>0$ and for  any $x_1,x_2 \in \RR$.
Then, using a variation of constants formula,
 see e.g.~\cite[Section 3.3]{Co65}  or~\cite[Theorem 2.25]{Jsell}, we prove that the set, cf.~\eqref{WueWs},
\begin{equation}\label{charac}
\tilde{W}^u_{l_u}(\tau):=  \Big\{ \Q  \mid   \lito \boldsymbol{x_{l_u}}(t,\tau;\Q) = (0,0) \Big\}
\end{equation}
is a graph on $\ell^u(\tau)$ (globally),
 for any $\tau \in \RR$. Then the proof
follows from a truncation argument.

Using the flow of~\eqref{si.hardy} we get the following.
\begin{lemma}\label{global}
Assume $\Humeno$ and \textbf{C}. Then the set $\tilde{W}^u_{l_u}(\tau)$
characterized as in~\eqref{charac}
 is a $1$-dimensional immersed submanifold having $\ell^u(\tau)$
as tangent space in the origin.
\end{lemma}
\begin{proof}
Let us denote by $\Phi_{T,\tau}$ the diffeomorphism induced by the flow of~\eqref{si.hardy}:
i.e. $\Phi_{T,\tau}(\Q)=\boldsymbol{x_{l_u}}(T,\tau;\Q)$.
Then $\Phi_{T,\tau}(W^u_{l_u,loc}(\tau))$ is a $1$-dimensional submanifold for any $\tau,T \in \RR$ and
$\Phi_{T,\tau_1}(W^u_{l_u,loc}(\tau_1)) \supset \Phi_{T,\tau_2}(W^u_{l_u,loc}(\tau_2))$ if $\tau_1< \tau_2$.
Hence we may set $\tilde{W}^u_{l_u}(T):=\bigcup_{\tau\in\RR} \Phi_{T,\tau}(W^u_{l_u,loc}(\tau))$
and we see that $\tilde{W}^u_{l_u}(T)$ is a $1$-dimensional immersed manifold, and by construction it is characterized as in~\eqref{charac}.
\end{proof}

\begin{remark}\label{8shaped}
We stress that $\tilde{W}^u_{l_u}(\tau)$ (and $\tilde{W}^s_{l_s}(\tau)$  constructed below)
 may be not a usual submanifold in the origin: i.e it may be $8$ shaped as in the critical autonomous case, see e.g. Figure~\ref{livelli}.
However it always contains $W^u_{l_u,loc}(\tau)$ (respectively ${W}^s_{l_s, loc}(\tau)$) which is tangent to $\ell^u(\tau)$
(resp. $\ell^s(\tau)$).
\end{remark}

Now we drop the assumption \textbf{C} and we prove Lemma~\ref{corrW}.

\medbreak

\noindent \emph{Proof of Lemma~\ref{corrW}}. \label{proofcorrW}
Fix $\tau \in\RR$; for every $\Q\in\RR^2$ we can introduce
$$
\mathfrak T(\Q,\tau)=\sup \big\{t \mid \boldsymbol{x_{l_u}}(\cdot,\tau,\Q) \text{ is defined in }[\tau, t) \big\} \,.
$$
Then $\lim_{t\to \mathfrak T(\Q,\tau)} |\boldsymbol{x_{l_u}}(t,\tau,\Q)|=+\infty$ if $\mathfrak T(\Q,\tau)<+\infty$. It is easy to verify that $\mathfrak T(\cdot,\tau)$ is lower semicontinuous, i.e.  the sets
$ \{ \Q \in \RR^2 \mid \mathfrak{T}(\Q,\tau) \le \mathfrak t \}$ are closed.
In fact for every $\Q_0\in\RR^2$ and every $\ep >0$ we can find neighbourhoods $\mathcal U$ of $\Q_0$ and $\mathcal V$ of $\boldsymbol{x_{l_u}}(\mathfrak T(\Q_0,\tau) -\ep,\tau,\Q_0)$ such that
for every $\Q\in \mathcal U$ we have $\boldsymbol{x_{l_u}}(\mathfrak T(\Q_0,\tau) -\varepsilon,\tau,\Q)\in \mathcal V$, thus giving
$\mathfrak T(\Q,\tau) >  \mathfrak T(\Q_0,\tau)-\epsilon$ for every $\Q\in\mathcal U$.
Therefore if $\bs{Q_n} \to \Q_0$, then $\liminf_{n \to \infty}
\mathfrak T(\Q_n,\tau) \ge \mathfrak T(\Q_0,\tau)$.

 Then, we consider
\begin{equation}\label{defTu}
\mathfrak{T}^u(\tau):= \inf \big\{ \mathfrak{T}(\Q,\tau) \mid \Q \in W^{u,+}_{l_u,loc}(\tau) \big\}\,.
\end{equation}
Notice that $\mathfrak T ((0,0),\tau)=+\infty$ and that $\mathfrak{T}^u(\tau)$ is increasing by construction. The lower semicontinuity gives us that either one has $\mathfrak{T}^u(\tau)=+\infty$ or the infimum is in fact a minimum, being $W^{u,+}_{l_u,loc}(\tau)$ bounded and $\mathfrak{T}(\Q,\tau) > \tau$. Moreover the subset containing the points $\Q$ which {\em explode to infinity} before a fixed time $\mathfrak t$,
$$
\mathcal X(\mathfrak t,\tau) = \big\{ \Q \in W^{u,+}_{l_u,loc}(\tau) \mid \mathfrak{T}(\Q,\tau) \leq \mathfrak t \big\}\,,
$$
is a closed subset. Conversely $\mathcal W(\mathfrak t,\tau) = W^{u,+}_{l_u,loc}(\tau) \setminus \mathcal X(\mathfrak t,\tau)$ is a relatively open subset
of $W^{u,+}_{l_u,loc}(\tau)$.

If $\mathfrak t<\mathfrak{T}^u(\tau)$, then $\mathcal X(\mathfrak t,\tau)=\emptyset$, so that
$\tilde{W}^{u,+}_{l_u}(\mathfrak t):=\Phi_{\mathfrak t,\tau}W^{u,+}_{l_u,loc}(\tau)$
 is diffeomorph to  $W^{u,+}_{l_u,loc}(\tau)$, and it is a $1$-dimensional manifold with border.
 In fact it is easy to check that
 the map $\Phi_{\tau,T}= \Phi_{T,\tau}^{-1}$ is well defined in an open neighborhood of $\Phi_{T,\tau}W^{u,+}_{l_u,loc}(\tau)$.

t
\begin{figure}[t]
\centerline{\epsfig{file=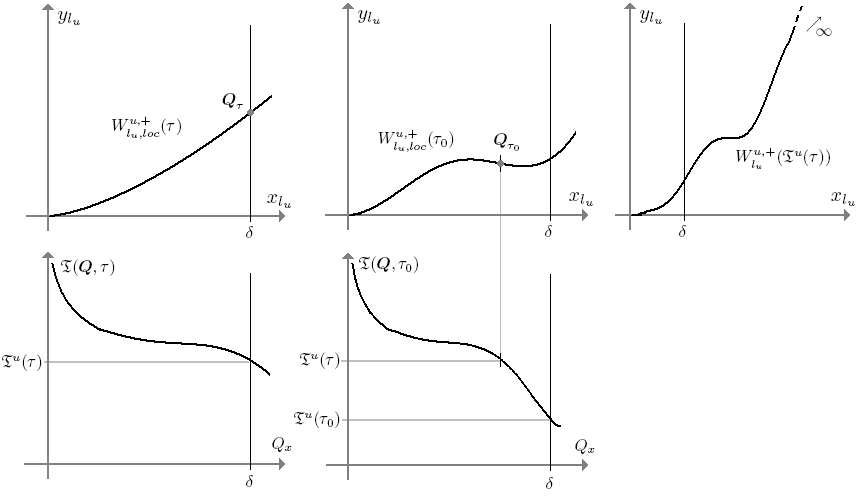, width = 12 cm}}
\caption{Assume that  $\mathfrak{T}(\Q,\tau)$ is continuous and strictly decreasing in $Q_x$. At the time $\tau$ (on the left), the endpoint $\Q_\tau=(\delta,Q_y)$ minimizes $\mathfrak{T}(\cdot,\tau)$ along ${W}^{u,+}_{l_u,loc}(\tau)$, thus having $\mathcal W(\mathfrak{T}^u(\tau),\tau)={W}^{u,+}_{l_u,loc}(\tau)\setminus\{\Q_\tau\}$. If we consider $\tau_0 < \tau$ (at the center), then $\mathfrak{T}^u(\tau_0)<\mathfrak{T}^u(\tau)$ and the set $\mathcal W(\mathfrak{T}^u(\tau),\tau_0)$ consists of the points to the left with respect to $\Q_{\tau_0}=\boldsymbol{x_{l_u}}(\tau_0,\tau,\Q_\tau)$. The images $\Phi_{\mathfrak{T}^u(\tau),\tau} \mathcal W(\mathfrak{T}^u(\tau),\tau)$ and $\Phi_{\mathfrak{T}^u(\tau),\tau_0}\mathcal W(\mathfrak{T}^u(\tau),\tau_0)$ gives us the unbounded $1$-dimensional manifold $W^{u,+}_{l_u}(\mathfrak{T}^u(\tau))$ (on the right).
}
\label{figureWdec}
\end{figure}

 We assume first for illustrative purpose that, \emph{for any $\Q=(Q_x,Q_y)\in W^{u,+}_{l_u,loc}(\tau)$
 the function $\mathfrak{T}(\Q,\tau)$ is strictly decreasing in $Q_x$}:
this is the case, e.g., if~\eqref{si.hardy} is autonomous. This assumption will be removed later on.

If we set $\mathfrak t = \mathfrak{T}^u(\tau)$ we have $\mathcal X(\mathfrak{t},\tau)=\{\Q_{\tau}\}$ where $\Q_{\tau}=(\delta,Q_y)$ is the endpoint of ${W}^{u,+}_{l_u,loc}(\tau)$, while if $\mathfrak t > \mathfrak{T}^u(\tau)$ the sets $\mathcal X(\mathfrak{t},\tau)$, which contains $\Q_\tau$, and $\mathcal W(\mathfrak{t},\tau)$ are connected. In both the cases, the map $\Phi_{\mathfrak{t},\tau}$ is well defined on  $\mathcal W(\mathfrak{t},\tau)$, and
the set $\tilde{W}^{u,+}_{l_u}(\mathfrak{t}):=\Phi_{\mathfrak{t},\tau}\mathcal W(\mathfrak{t},\tau)$ is  diffeomorph  to $\mathcal W(\mathfrak{t},\tau)$.
   In fact there is an open neighborhood of $\tilde{W}^{u,+}_{l_u}(\mathfrak{t})$ which is mapped by the inverse diffeomorphism
   $\Phi_{\tau,\mathfrak{t}}$ into an open neighborhood of $\mathcal W(\mathfrak{t},\tau)$.
   So $\tilde{W}^{u,+}_{l_u,loc}(   \mathfrak{t})\setminus \{(0,0) \}$ is a $1$-dimensional manifold without border, see Figure~\ref{figureWdec}.

      Now let us repeat the discussion replacing $\tau$ by $\tau_0<\tau$.
   It is easy to check that $\Phi_{\mathfrak t,\tau_0}\mathcal W(\mathfrak{t},\tau_0) \supseteq
   \Phi_{\mathfrak t,\tau}\mathcal W(\mathfrak{t},\tau)$,
    and if  $\mathfrak t \ge \mathfrak{T}^u(\tau)$, then
   $\Phi_{\mathfrak t,\tau_0}\mathcal W(\mathfrak{t},\tau_0) =
   \Phi_{\mathfrak t,\tau}\mathcal W(\mathfrak{t},\tau)=\tilde{W}^{u,+}_{l_u}(\mathfrak t)$, and it is unbounded.
   Furthermore by construction
\begin{equation}\label{property}
\tilde{W}^{u,+}_{l_u}(\mathfrak t) = \big\{ \Q \mid \lito \boldsymbol{x_{l_u}}(t,\mathfrak t,\Q)=(0,0) \, ,
\,\, \dot{x}_{l_u}(t,\mathfrak t,\Q)>0 \textrm{ for $t \ll 0$} \big\}\,,
\end{equation}
if $\mathfrak t \ge \mathfrak{T}^u(\tau)$.
Therefore the set
$$
W^{u,+}_{l_u}(T)= \bigcup_{\tau \leq T}  \Phi_{T,\tau}\mathcal W(T,\tau)
$$
is characterized by the property defined in~\eqref{property}.
\begin{figure}[t]
\centerline{\epsfig{file=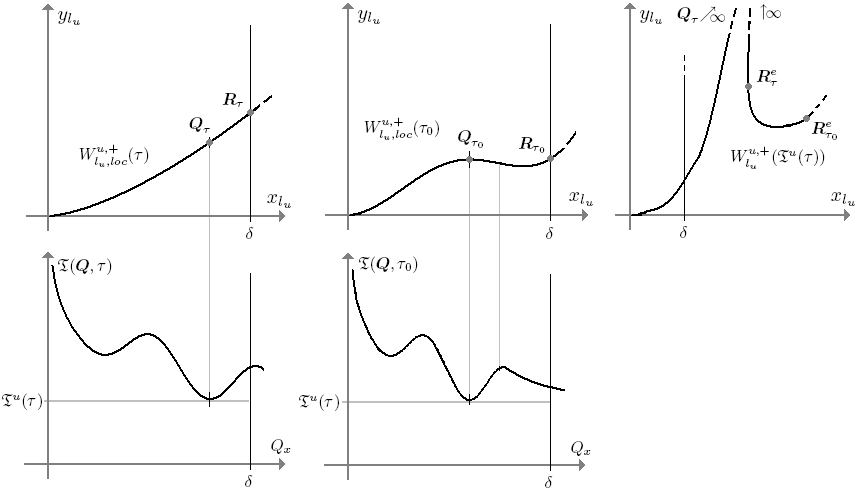, width = 12 cm}}
\caption{If $\mathfrak{T}(\Q,\tau)$ is not decreasing in $Q_x$, at the time $\tau$ (on the left), the minimum $\mathfrak{T}^u(\tau)$ can be attained in a point  $\Q_\tau=(Q_x^\tau,Q_y^\tau)$ with $ Q_x^\tau<\delta$, while we denote by $\R_\tau$ the endpoint of $W^{u,+}_{l_u,loc}(\tau)$. We consider also $\tau_0 < \tau$ (in the center), where we can find the point $\Q_{\tau_0}=\boldsymbol{x_{l_u}}(\tau_0,\tau,\Q_\tau)$ and we denote by $\R_{\tau_0}$ the endpoint of $W^{u,+}_{l_u,loc}(\tau_0)$. In both the situations the sets $\mathcal W(\mathfrak{T}^u(\tau),\tau)$ and $\mathcal W(\mathfrak{T}^u(\tau),\tau_0)$ are disconnected respectively at the point $\Q_\tau$ and $\Q_{\tau_0}$. The images $\Phi_{\mathfrak{T}^u(\tau),\tau} \mathcal W(\mathfrak{T}^u(\tau),\tau)\subset\Phi_{\mathfrak{T}^u(\tau),\tau_0}\mathcal W(\mathfrak{T}^u(\tau),\tau_0)$ gives us two unbounded disconnected sets contained in $\tilde{W}^{u,+}_{l_u}(\mathfrak{T}^u(\tau))$ (on the right). The second components have endpoints respectively $\R_\tau^e = \Phi_{\mathfrak{T}^u(\tau),\tau}(\R_\tau)$ and $ \R_{\tau_0}^e = \Phi_{\mathfrak{T}^u(\tau),\tau_0}(\R_{\tau_0})$; while the point $\Q_\tau$ is, roughly speaking, sent to infinity by the flux $\Phi_{\mathfrak{T}^u(\tau),\tau}$.
}
\label{figureWnondec}
\end{figure}

\begin{figure}[t]
\centerline{\epsfig{file=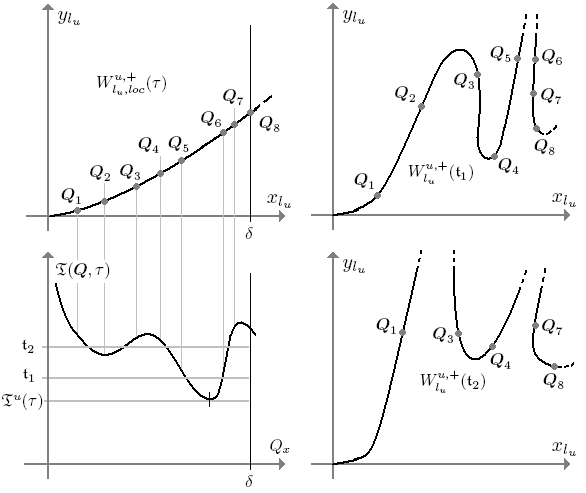, width = 10 cm}}
\caption{
In general $\mathcal W(\mathfrak t,\tau)\subset W^{u,+}_{l_u,loc}(\tau)$ may be disconnected when $\mathfrak t> \mathfrak{T}^u(\tau)$.
The picture sketches an example. Consider $\mathfrak t_2> \mathfrak t_1 >\mathfrak{T}^u(\tau)$, the set $\mathcal W(\mathfrak t_1,\tau)$ has two connected components, while $\mathcal W(\mathfrak t_2,\tau)$ has three components. On the right we have drawn the corresponding images $\Phi_{\mathfrak t_1,\tau}\mathcal W(\mathfrak t_1,\tau)\subset W^{u,+}_{l_u}(\mathfrak t_1)$ and $\Phi_{\mathfrak t_2,\tau}\mathcal W(\mathfrak t_2,\tau)\subset W^{u,+}_{l_u}(\mathfrak t_2)$.
We show how some points $\Q_1,\ldots,\Q_8$, are mapped by the fluxes $\Phi_{\mathfrak t_1,\tau}$ and $\Phi_{\mathfrak t_2,\tau}$, denoting the images again with $\Q_1,\ldots,\Q_8$ for simplicity. In particular, at the time $\mathfrak t_2$ the solution $\boldsymbol{x_{l_u}}(\cdot,\tau,\Q_j)$ is not defined for $j=2,5,6$.
}
\label{figureWdisc}
\end{figure}

If we remove the simplifying assumption   that $\mathfrak{T}(\bs Q ,\tau)$
 is decreasing with respect to $Q_x$,
we can repeat the previous discussion with the following changes.
 The open set $\mathcal W(\mathfrak{T}^u(\tau),\tau)$ can be disconnected, so its image $\Phi_{\mathfrak{T}^u(\tau),\tau}\mathcal W(\mathfrak{T}^u(\tau),\tau)$ may be disconnected too, see Figure~\ref{figureWnondec}. However $\mathcal W(\mathfrak{T}^u(\tau),\tau)$ has a connected component containing the origin, say $\mathcal W_1(\mathfrak{T}^u(\tau),\tau)$, whose image is the connected component of $\Phi_{\mathfrak{T}^u(\tau),\tau}\mathcal W(\mathfrak{T}^u(\tau),\tau)$ containing the origin. Observe that $\mathcal W_1(\mathfrak{T}^u(\tau),\tau)$ is a connected one dimensional manifold, so this property is
 inherited by $\Phi_{\mathfrak{T}^u(\tau),\tau}\mathcal W(\mathfrak{T}^u(\tau),\tau)$ too.

When $\mathfrak t  \ge \mathfrak{T}^u(\tau)$,
the map $\Phi_{\mathfrak{t},\tau}$ is well defined in $\mathcal W(\mathfrak{t},\tau)$ and the set  $\mathcal X(\mathfrak{t},\tau)$ may disconnect $\mathcal W(\mathfrak{T}^u(\tau),\tau)$, see Figure~\ref{figureWdisc}. Repeating the previous arguments we see that the image
$\Phi_{\mathfrak{t},\tau}\mathcal W(\mathfrak{t},\tau)$ is unbounded and may have many components. However, $\mathcal W(\mathfrak{t},\tau)$ has a connected component containing the origin, say $\mathcal W_1(\mathfrak{t},\tau)$ and we can consider the image $W^{u,+}_{l_u}(\mathfrak t):=\Phi_{\mathfrak t,\tau}\mathcal W_1(\mathfrak{t},\tau)$ which is a $1$-dimensional connected manifold and it is unbounded.

Again, cf. Figure~\ref{figureWnondec}, if we switch from $\tau$ to $\tau_0<\tau$ we see that
    $\Phi_{\mathfrak t,\tau_0}\mathcal W_1(\mathfrak{t},\tau_0) \supseteq
   \Phi_{\mathfrak t,\tau}\mathcal W_1(\mathfrak{t},\tau)$,
    and if  $T = \mathfrak{T}^u(\tau)$, then
$$
W^{u,+}_{l_u}(T)=\Phi_{T,\tau}\mathcal W_1(T,\tau)=\Phi_{T,\tau_0}\mathcal W_1(T,\tau_0) \supseteq \Phi_{T,\tau_1} \mathcal W_1(T,\tau_1)
$$
for any $\tau_0<\tau<\tau_1$. Hence, we can define for every $T\in\RR$ the set
$$
W^{u,+}_{l_u}(T)= \bigcup_{\tau \leq T} \Phi_{T,\tau}\mathcal W_1(T,\tau)
$$
which is a $1$-dimensional connected manifold containing the origin in its border, and it is unbounded.
Reasoning in the same way we see that if $\tau_0 < \tau_1$ then
$$\Phi_{T,\tau_0}\mathcal W(T,\tau_0) \supseteq \Phi_{T,\tau_1} \mathcal W(T,\tau_1) \, ,$$
therefore we can define
$$\tilde{W}^{u,+}_{l_u}(T):= \bigcup_{\tau \leq T} \Phi_{T,\tau}\mathcal W(T,\tau) $$
Notice  that $\tilde{W}^{u,+}_{l_u}(T)$ may not be a manifold (it may break infinitely many times),
since $\mathcal W(T,\tau) $ may be disconnected.
However by construction $\tilde{W}^{u,+}_{l_u}(T)$ may be still characterized as
in \eqref{property}; moreover
 $W^{u,+}_{l_u}(T)$ is the connected component of $ \tilde{W}^{u,+}_{l_u}(T)$
containing the origin in its border, and it is, as shown above, a $1$-dimensional connected manifold.

The construction of $W^{u,-}_{l_u}(\tau)$ and of $W^{u}_{l_u}(\tau)=W^{u,-}_{l_u}(\tau) \cup W^{u,+}_{l_u}(\tau)$
  is completely analogous and it is omitted.
  This concludes the  part of the proof of Lemma~\ref{corrW}
concerning the unstable manifolds. The construction of the stable leaves
is very similar and we just  sketch it.

With a specular argument we assume $\Hsmeno$, so that
  $\mathcal{A}_{l_s}(+\infty)$ has $\nu_2<0<\nu_1$
as eigenvalues, where $\nu_1:=\alpha_{l_u}-\kappa(\beta)$ and $\nu_2:=\alpha_{l_s}+2-n+\kappa(\beta)$. So, let
$Y(t)$ be the fundamental matrix of
\eqref{eqmatriciale}, where $\mathcal{A}_{l_u}(t)$ is replaced by $\mathcal{A}_{l_s}(t)$.
Then,
for any $\tau \in \RR$ there
is a  constant $K=K(\tau)>1$,   and a projection
$\mathcal P^+$ such that
\begin{equation}\label{dichots}
\begin{array}{cc}
\|Y(t)(I-\mathcal P^+) Y(s)^{-1} \| \le K \eu^{\nu_1 (t-s)} & \textrm{ for any $s>t>\tau$} \,,\\
\|Y(t)\mathcal P^+ Y(s)^{-1} \| \le K \eu^{\nu_2 (t-s)}  & \textrm{ for any $t>s>\tau$} \,,
\end{array}
\end{equation}
see again~\cite[Section 4]{Co78}, and~\cite[Appendix]{CF}.
Denote by $\mathcal P^+(\tau):=Y(\tau)\mathcal P^+ Y(\tau)^{-1}$, and by
 $\ell^s(\tau)$ the $1$-dimensional  range of $\mathcal P^+(\tau)$.
Then the solution
$\vec{\xi}(t)$ of~\eqref{eqmatriciale}, with $l_s$ replacing $l_u$, is bounded for $t \ge 0$ iff
 $\vec{\xi}(0) \in \ell^s(\tau)$. Moreover $\|\vec{\xi}(t)\| \eu^{-\nu_2 t} \to c$ as $t \to +\infty$
 for a suitable $c>0$.
This way we are able to construct a local manifold $W^s_{l_s,loc}(\tau)$ and
to reprove a result analogous to
Lemma~\ref{loc}. Then, assuming temporarily $\textbf{C}$ and reasoning as in Lemma~\ref{global},
we see that $\Phi_{T,\tau}(W^s_{l_s,loc}(\tau))$ is a $1$-dimensional submanifold for any $\tau,T \in \RR$; moreover
$\Phi_{T,\tau_2}(W^s_{l_s,loc}(\tau_2)) \supset \Phi_{T,\tau_1}(W^s_{l_s,loc}(\tau_1))$ if $\tau_1< \tau_2$.
Hence, assuming $\textbf{C}$ and $\Hsmeno$, we obtain that the set
\begin{equation}\label{characs}
\tilde{W}^s_{l_s}(\tau):=\bigcup_{\tau_0\ge \tau}   \Phi_{\tau,\tau_0}(W^s_{l_s,loc}(\tau_0))
 = \big\{ \Q \mid \lit \boldsymbol{x_{l_s}}(t,\tau,\Q)=(0,0) \big\}
\end{equation}
 is a $1$-dimensional immersed manifold having $\ell^s(\tau)$
as tangent space in the origin.

Then we remove assumption \textbf{C} and, arguing as above, we see that $\tilde{W}^s_{l_s}(\tau)$ may be disconnected, but
its connected component containing the origin, denoted by $W^s_{l_s}(\tau)$, is again a $1$-dimensional manifold.
Then repeating the previous discussion we conclude the proof of Lemma~\ref{corrW}. The part of the proof concerning Lemmas~\ref{corrispondenze} and
 \ref{corrispondenze.bis} is given below. \qed
\medbreak

Now we proceed with the proof of Lemma~\ref{corrispondenze.bis}, which includes   Lemma~\ref{corrispondenze} as  a
particular case. The proof is adapted from \cite[Lemma 2.10]{DF} where it is developed assuming \textbf{C} and $h(r) \equiv 0$.

\medbreak

\noindent  \label{proofcorr}
\emph{Proof of Lemma~\ref{corrispondenze.bis}}\\
 Assume  $\Humeno$ and $\Hsmeno$; recalling that
  $\boldsymbol{x_{l_u}}(t)=(u(\eu^t) \eu^{\alpha_{l_u} t},u'(\eu^t) \eu^{(1+\alpha_{l_u}) t})$, we find
  $\boldsymbol{x_{l_s}}(t) = \boldsymbol{x_{l_u}}(t) \eu^{(\alpha_{l_s}-\alpha_{l_u})t}$.
  Therefore in particular    $\R=\Q \textrm{exp}[ -(\alpha_{l_u}-\alpha_{l_s}) \tau]$.

Assume first \textbf{C} for simplicity.
  From roughness of exponential dichotomy, cf~\cite[Chapter 4]{Co78} and  \cite[Theorem 2.16]{Jsell},  we see that, if $\Q \in W^u_{l_u}(\tau)$, then
  there is $d=d(\Q)$ such that $\lim_{t\to-\infty} \boldsymbol{x_{l_u}}(t,\tau,\Q) \eu^{-[\alpha_{l_u}-\kappa(\eta)]t} =d(1,-\kappa(\eta))$.
  Assume
  $d>0$ for definiteness;
   then for the corresponding solution $u(r)$ of~\eqref{hardy}     we get
  \begin{equation}\label{stimau}
    u(r)={x_{l_u}}(\ln(r),\tau,\Q) r^{-\alpha_{l_u}+\kappa(\eta)} \to d \,, \quad \textrm{as $r \to 0$} \,.
\end{equation}
Assume now $\Q \not\in W^u_{l_u}(\tau)$.
    Then, if $l_u  \ne 2^*$, we find that $|\boldsymbol{x_{l_u}}(t,\tau,\Q)|$ is uniformly positive as $t \to -\infty$, and if $l_u = 2^*$
      there is a sequence $t_n \to -\infty$ such that $|\boldsymbol{x_{l_u}}(t_n,\tau,\Q)|$ is uniformly positive:
    in both the cases   the corresponding solution $u(r)$  of~\eqref{hardy} is not a \Rsol-solution since  $u(r)r^{\alpha_{l_u}-\kappa(\eta)} \not\to 0$
    as $r \to 0$.
    Further we easily see that $\Q \in W^u_{l_u}(\tau) \iff \R \in W^u_{l_s}(\tau) $.

  Arguing similarly, if $\Q \in W^s_{l_s}(\tau)$ then
  there exists $L=L(\Q)>0$ such that $\lim_{t\to+\infty} \boldsymbol{x_{l_s}}(t,\tau,\Q) \eu^{-[\gamma_{l_s}+\kappa(\beta)]t} = L(1,-(n-2)+\kappa(\beta))$:
  hence the corresponding solution  $u(r)$ of~(\ref{eq.na}) satisfies
    \begin{equation}\label{stimas}
 \lim_{t\to+\infty} u(\eu^t)\eu^{(\alpha_{l_s}-\gamma_{l_s}-\kappa(\beta))t}= \lim_{r\to+\infty}
  u(r)r^{n-2-\kappa(\beta)} = L\,.
 \end{equation}
  So we can easily conclude as above.

Hence, if we assume either $\Hupiu$ or $\Humeno$, we can construct
the unstable manifold  $W^u_{l_u}(\tau)$ for any $\tau \in \RR$; similarly if either
$\Hspiu$ or $\Hsmeno$ hold, we can construct the stable manifold  $W^s_{l_s}(\tau)$ for any $\tau \in \RR$.
Moreover Remark~\ref{allinfinito} still holds and we can construct $W^s_{l_u}(\tau)$ and
$W^u_{l_s}(\tau)$ via~\eqref{cambioL} too.

Now we drop assumption \textbf{C}.
In this case, due to the presence of non-continuable trajectories, we need to distinguish between $W^u_{l_u}(\tau)$
and $\tilde{W}^u_{l_u}(\tau)$, and similarly for the other manifolds.
In fact $\bs{x_{l_u}}(t,\tau,\Q) \to (0,0)$ iff $\Q \in \tilde{W}^u_{l_u}(\tau)$. Further for any $\Q \in \tilde{W}^u_{l_u}(\tau)$
we can find $N \gg 1$ such that $\bs{x_{l_u}}(T,\tau,\Q)\in W^u_{l_u}(T)$ for any $T \le -N$. So we can repeat the previous argument and
we see that the corresponding
solution $u(r)$ of \eqref{hardy} is a \Rsol-solution. A similar argument holds for the stable manifold.

So for any $\tau$ we   find that $\Q \in W^u_{l_u}(\tau)$ iff  $\R \in W^u_{l_s}(\tau)$ iff $u(r)=u(r,d)$ is a \Rsol-solution
with $0<d<d^+_\tau$, see \eqref{eqrhod}. Similarly $\Q \in W^s_{l_u}(\tau)$ iff  $\R \in W^s_{l_s}(\tau)$ if $u(r)=u(r,L)$ is a \fdsol-solution
with $0<L<L^+_\tau$.
\qed

\subsection{Proof of Lemmas~\ref{paramreg} and~\ref{thetaphiS}}\label{app2}

We prove now Lemma~\ref{paramreg}: such a result has been obtained in presence of continuability of the solutions and for $h(r) \equiv 0$ in~\cite[Lemma 2.10]{DF}.

\medbreak

\noindent
\emph{Proof of Lemma~\ref{paramreg}.}
We will prove only the first part of the statement, the second follows similarly.
Consider the parametrization $\Sigma^{u,+}_{l_u}(\cdot,T)$.
Assume first~\textbf{C}.
Observe that, starting from $\Sigma^{u,+}_{l_u}(\cdot,T)$, we can construct a parametrization of
 $W^{u}_{l_u}(\tau)$ for any $\tau \in \RR$, by setting $\Sigma^{u,+}_{l_u}(\omega,\tau):= \boldsymbol{x_{l_u}}(\tau;T,\Sigma^{u,+}_{l_u}(\omega,T)) $. In fact, the function
$\Sigma^{u,+}_{l_u}: [0,+\infty) \times \RR \to \RR^2$ is continuous in both the variables,
and the map $(\omega,\tau) \mapsto (\Sigma^{u,+}_{l_u}(\omega,\tau), z(\tau))$ is injective in  $\boldsymbol{W^{u}}$.
According to this parametrization, $\boldsymbol{x_{l_u}}(t;\tau,\Sigma^{u,+}_{l_u}(\omega,\tau))$  coincides with
$\boldsymbol{x_{l_u}}(t;T,\Sigma^{u,+}_{l_u}(\omega,T))$ and
corresponds to the given solution $u(r,d(\omega))$ for any $\tau \in \RR$.

Fix $N\in\RR$ and let $\delta:=\delta(N)$ be the constant defined in Lemma~\ref{loc};
 we can find $\bar\omega>0$ and  $N(\bar{\omega})<N$ such that
$\Sigma^{u,+}_{l_u}(\omega,\tau) \in W^{u,+}_{l_u, \textrm{loc}}(\tau)$, whenever $0\le \omega \le \bar{\omega}$ and $\tau \le N(\bar{\omega})$.

We now show that $d(\omega)$ is strictly increasing.
Once proved this claim for this particular parametrization we have it
for any parametrization $ \varpi \to \Sigma^{u,+}_{l_u}(\varpi,\tau)$ of
$W^{u,+}_{l_u}(\tau)$ as in the assumption of Lemma~\ref{paramreg}, due to the
monotonicity of the change of variables $\varpi(\omega)$.
 Using Lemma~\ref{loc} we see that we can choose $\omega_1<\omega_2$,
 so that
  $\Sigma^{u,+}_{l_u}(\omega_i,\tau) \in W^{u,+}_{l_u, \textrm{loc}}(\tau)$ for any $\tau \le N(\omega_2)$ and for $i=1,2$.
  Hence
  $\Sigma^{u,+}_{l_u}([0,\omega_2]\times\{\tau\})\subset W^{u,+}_{l_u}(\tau)$ is a graph on $\ell^u(\tau)$,
for any $\tau  \le N(\omega_2)$, see Lemma~\ref{loc}.
In particular $x_{l_u}(t; T,\Sigma^{u,+}_{l_u}(\omega_1,T))-x_{l_u}(t; T,\Sigma^{u,+}_{l_u}(\omega_2,T))<0$ for $t=N(\omega_2)$.
We claim that \emph{$W^{u,+}_{l_u,loc}(\tilde{\tau})$ is a graph on a segment of the $x$ axis, for any $\tilde\tau  \le N(\omega_2)$}.
In fact when $\eta=0$ the claim is obvious since $\ell^u(\tau)$ is contained in the $x$ axis.
If $\eta \ne 0$, since $\ell^u(\tau)$ is not orthogonal to the $x$ axis,
possibly choosing a smaller $\delta$   we can again assume that  $W^{u,+}_{l_u,loc}(\tilde{\tau})$
is a graph on the $x$ axis too, so the claim is true.

Assume for contradiction that $d(\omega_1)>d(\omega_2)$,
 then
there is $\tilde{\tau}<N(\omega_2)$ such that
$x_{l_u}(t; T,\Sigma^{u,+}_{l_u}(\omega_1,T))-x_{l_u}(t; T,\Sigma^{u,+}_{l_u}(\omega_2,T))$
is positive
for any $t<\tilde{\tau}$ and it is zero for $t=\tilde{\tau}$. In particular $W^{u,+}_{l_u,loc}(\tilde{\tau})$
is not a graph on the $x$ axis, so we have found a contradiction.  Hence $d(\omega_1)<d(\omega_2)$, and the Lemma is concluded if \textbf{C} holds.
Notice that we can redefine the parametrization and use directly $d$ instead of $\omega$ as parameter, so,
with a little abuse of notation we find the parametrization $\Sigma^{u,+}_{l_u}(d,\tau)$ of  $W^{u,+}_{l_u}(\tau)$ which is continuous (and smooth)
 in both the variables
for any $(d,\tau) \in [0,+\infty) \times \RR$.

Now we drop \textbf{C}. Fix $T \in \RR$, and correspondingly $d^+_T$ as in \eqref{eqrhod},
so that,   for any $d\in(0,d^+_T$),  $u(r,d)$ is continuable for any $0<r< \eu^{T}$.
Using the previous discussion and a truncation argument,
for any $D\in(0,d^+_T)$, we can define the map
$\Sigma^{u,+}_{l_u}(d,T)$ for $d \in [0,D]$;
 so we get a
parametrization of a connected branch of  $W^{u,+}_{l_u}(T)$, say $\bar{W}^{u,+}_{l_u}(T)$.
Since for any point $\Q \in \bar{W}^{u,+}_{l_u}(T)$ we have that $\bs{x_{l_u}}(\tau; T, \Q)$
exists for any $\tau\le T$, arguing as above, we find that
$\Sigma^{u,+}_{l_u}(d,\tau)= \boldsymbol{x_{l_u}}(\tau,T,\Sigma^{u,+}_{l_u}(d,T)) $ is a parametrization of a connected branch of  $W^{u,+}_{l_u}(\tau)$, denoted again by
$\bar{W}^{u,+}_{l_u}(\tau)$, for any $\tau \le T$.

Now let $\tau>T$
and notice that the function $d^+(\tau)=d^+_{\tau}$ defined in \eqref{eqrhod} is decreasing in $\tau$.  If $D < d^+_{\tau}$, reasoning as above,
 we find that $\Sigma^{u,+}_{l_u}(d,\tau)= \boldsymbol{x_{l_u}}(\tau,T,\Sigma^{u,+}_{l_u}(d,T)) $ gives again a parametrization of a connected branch of $W^{u,+}_{l_u}(\tau)$, for $0\le d \le D$.
 If $D \ge d^+_{\tau}$ then  $\Sigma^{u,+}_{l_u}(d,\tau)= \boldsymbol{x_{l_u}}(\tau,T,\Sigma^{u,+}_{l_u}(d,T)) $  for $0< d< d^+_{\tau}$ is unbounded and it is itself a parametrization of
 the whole manifold $W^{u,+}_{l_u}(\tau)$.  Assume $D < d^+_{T}$,
 this way we  obtain  a map $\Sigma^{u,+}_{l_u}(d,\tau)$ which is continuous
 in both the variables for
  $$\bar{E}= \big([0,D) \times (-\infty, T] \cup \{(d,\tau) \mid 0<d<\max\{D,d^+_{T}\} \,, \; \tau \ge T \} \big) .$$
 Further $d \to \Sigma^{u,+}_{l_u}(d,\tau)$ is injective for $(d,\tau) \in  \bar{E}$.
 For the arbitrariness of $D<d^+_T$ we can let $D \to d^+_T$. Then from the arbitrariness
 of $T \in \RR$  we define
   $\Sigma^{u,+}_{l_u}$ in the whole
$E =\{(d,\tau) \mid 0<d< d^+_{\tau} \, , \;  \tau \in \RR \} $,    it is continuous in both the variables and it gives a bijective parametrization
of $W^{u,+}_{l_u}(\tau)$ for any $\tau \in \RR$. \qed

\medbreak

Now we prove Lemma~\ref{thetaphiS}: the argument is a modification of~\cite[Propositions~3.5,~3.8]{DF}.
  In fact, in this setting,
 we need to take into account the fact that $\ell^u(\tau)$ and $\ell^s(\tau)$
change with $\tau$ (due to the presence of Hardy potentials), while in~\cite{DF} there was not this difficulty.
In particular we need to ask for $l_s>2^*$ and to profit of Lemma~\ref{nongira}.

\medbreak

\emph{Proof of Lemma~\ref{thetaphiS}.}
We introduce some
definitions borrowed from~\cite{BDF,Fdie}.
Following~\cite{BDF,Fdie,JK},
given a curve $\boldsymbol{\gamma}:[a,b]\to\RR^2 \setminus \{(0,0)\}$, we define its rotation number $w(\boldsymbol{\gamma})$
by setting
\begin{equation}
\label{Rotnumb}
w(\boldsymbol{\gamma}):=\textrm{Int}\left[\frac{\theta_{\boldsymbol{\gamma}}(b)-\theta_{\boldsymbol{\gamma}}(a)}{2\pi}\right],
\end{equation}
where $\textrm{Int}[\cdot]$ denotes the integer part, and $\boldsymbol{\gamma}(t)=
(\rho_{\boldsymbol{\gamma}}(t)\cos \theta_{\boldsymbol{\gamma}}(t),
\rho_{\boldsymbol{\gamma}}(t)\sin \theta_{\boldsymbol{\gamma}}(t)).$
As pointed out in~\cite{Fdie}, we can extend this definition to a curve $\boldsymbol{\gamma}$ defined in
a semi-open interval $[a,b)$ if $\lim_{t \to b^-} \theta_{\boldsymbol{\gamma}}(t)$ exists (even if it is infinite).

Our argument will be rather sketchy since we just adapt~\cite{DF,Fdie}.
Let $\boldsymbol{\gamma^i}(t):[a,b] \to \RR^2$, for $i=1,2$, be curves in $\RR^2$  which do not
intersect each other, and let
$\phi(t)$ be a smooth monotone
function such as $\varphi(t)=z(t)=\eu^{\varpi t}$ as in~(\ref{si.naa}),
or $\varphi(t)=\zeta(t)=\eu^{-\varpi t}$ as in~(\ref{si.naas}) or
$\varphi(t)=t$ as in~\cite{BDF}.
Then
$\boldsymbol{\Gamma^i}(t)=(\boldsymbol{\gamma^i}(t),\varphi(t))$
 are  curves in $\RR^3$. Following
\cite{BDF}, we call linking number of
$\boldsymbol{\gamma^1},\boldsymbol{\gamma^2}$ in $[a,b]$ the number
$w(\boldsymbol{\gamma^1}-\boldsymbol{\gamma^2})$, i.e. the number of complete rotations
of a curve around the other. Such a quantity is invariant for homotopies in $\RR^3$
which preserve  the endpoints $\boldsymbol{\Gamma^1}(a)=\boldsymbol{\Gamma^2}(a)$
and $\boldsymbol{\Gamma^1}(b)=\boldsymbol{\Gamma^2}(b)$.

 We want to establish an homotopy between two curves so that linking number and rotation number are equal.
Let us fix $\tau \in \RR$ and $\Q\in W^{s,+}_{l_s}(\tau)$.
 Since $\boldsymbol{x_{l_s}}(t,\tau,\Q)$ converges to the origin as $t \to +\infty$ and $\dot{x}_{l_s}(t,\tau,\Q)<0$
 for $t \gg 1$, for every $\delta>0$ we can find $a=a(\delta) \gg 1$ such that  $| \bs{x_{l_s}}(a,\tau,\Q)|= \delta$
 and  $ |\bs{x_{l_s}}(t,\tau,\Q)|> \delta$ for $t \in [\tau, a)$.

 Then we set $\boldsymbol{\gamma^1}(t)=\boldsymbol{x_{l_s}}(t,\tau,\Q)$, where $\Q \in W^{s,+}_{l_s}(\tau)$,
and we consider the trajectory $\boldsymbol{\Gamma^1}(t)=(\boldsymbol{x_{l_s}}(t,\tau,\Q),\zeta(t))$ of~\eqref{si.naa}
for $t\in [\tau,a]$.\\
We shrink further  $\delta\leq \delta(\tau)$   so that the sets $W^s_{l_s, loc}(T)$ defined in Lemma~\ref{loc}
are graphs on $\ell^s(T)$ for any $T \ge \tau$, and we denote by $\boldsymbol{\bar{C}}(T)$
 the unique point in $W^{s,+}_{l_s, loc}(T)\cap \{x=\delta \}$.
 Let $0<L_a(T)<L_b$ be such that $\Upsilon^{s,+}_{l_s}(L_a(T),T)=\boldsymbol{\bar{C}}(T)$ and
 $\Upsilon^{s,+}_{l_s}(L_b,T)=\boldsymbol{x_{l_s}}(T,\tau,\Q)$.
  We consider the curves $\boldsymbol{\Psi_1}(T)=(\boldsymbol{\bar{C}}(T),\zeta(T))$ for $T \in [\tau,a]$,   the curve $\boldsymbol{\Psi_2}(d)=(\Upsilon^{s,+}_{l_s}(d,\tau),\zeta(\tau))$
 for $d \in [L_a(T), L_b]$ and the curve $\boldsymbol{\Gamma^2}(t)$ obtained following the graph of $\boldsymbol{\Psi_1}$ and then the graph of $\boldsymbol{\Psi_2}$.
 An homotopy between $\boldsymbol{\Gamma^1}$ and $\boldsymbol{\Gamma^2}$
 is obtained  by projecting
 $\boldsymbol{\Gamma^1}$ on $W^s_{l_s}(\tau)$ following the $2$-dimensional manifold $\boldsymbol{W^s}$: we sketch the construction,
 see~\cite[Lemma~4.3]{Fdie} for more details.

 For any $T \in [\tau,a]$ we construct the function $H0(\cdot,T)$ obtained following $\boldsymbol{\Psi_1}(S)=(\boldsymbol{\bar{C}}(S),\zeta(S))$ for $S\in [T,a]$,
 then
 $(\Upsilon^{s,+}_{l_s}(d,T),\zeta(T))$
 for $d \in [L_a(T), L_b]$ and  finally $\boldsymbol{\Gamma^1}(t)=(\boldsymbol{x_{l_s}}(t,\tau,\Q),\zeta(t))$ for $t \in [\tau,T]$
i.e.
 $$ H0(S,T)=\left\{  \begin{array}{ll}
 \boldsymbol{\Psi_1}(a+S (T-a)/L_a(T)) & \textrm{if }  S \in [a, L_a(T)] \\
    (\Upsilon^{s,+}_{l_s}(S,T),\zeta(T)) & \textrm{if }  S \in [L_a(T),L_b] \\
 (\boldsymbol{x_{l_s}}(S+(T-L_b)),\tau,\Q),\zeta( T) & \textrm{if }
                        S \in [L_b,\tau-(T-L_b)]
                      \end{array}
  \right. $$
Note that all the curves $S \mapsto H0(S,T)$ have the same endpoints and are homotopic; moreover $H0(\cdot,a)$ is $\boldsymbol{\Gamma^1}$ and
$H0(\cdot,\tau)$ is $\boldsymbol{\Gamma^2}$. So the linking number of
 $\boldsymbol{\Gamma^1}$ and $\boldsymbol{\Gamma^2}$ is $0$.

 Let us denote  by $\theta_{\Upsilon}(\tau,a)$,
 $\theta_{\bs{\bar{C}}}(\tau,a)$, $\theta_{\x}(\tau,a)$ respectively the angles performed by
  by $\Upsilon^{s,+}_{l_s}(L,\tau)$ for $L \in [L_a(T), L_b]$,
  by $\boldsymbol{\bar{C}}(T)$ for $T \in [\tau,a]$, and by
    $\boldsymbol{x_{l_s}}(t,\tau,\Q)$ for $t \in [\tau,a]$.
We have shown that
\begin{equation}\label{sommaangoli}
\theta_{\Upsilon}(\tau,a)+ \theta_{\bs{\bar{C}}}(\tau,a)= \theta_{\x}(\tau,a) \, .
\end{equation}
So Lemma~\ref{thetaphiS} follows from~\eqref{sommaangoli}, being~\eqref{angle} equivalent.
\qed

\subsection{On the Wazewski's principle}\label{app3}

We conclude the appendix  with a result, inspired by Wazewski's
principle, which allows to locate the unstable and the stable manifolds.
Consider
\begin{equation}\label{equazione}
    \dot{x}=F(x,t) \,,
\end{equation}
where $x \in \RR^2$, $t \in \RR$, $F$ continuous, and assume that the origin $\boldsymbol{O}=(0,0)$ is a critical point for~\eqref{equazione}.

Let $\mathcal{T}(\tau)$ be a closed set diffeomorphic to a full triangle.
We call the vertices $\boldsymbol{O}$, $\boldsymbol{A}(\tau)$ and $\boldsymbol{B}(\tau)$, and $o(\tau)$, $a(\tau)$, $b(\tau)$ the edges (without endpoints)
which are opposite to the respective vertex. Let $\mathcal{\hat{T}}(\tau)$ denote a further set diffeomorphic to a full triangle
having $\boldsymbol{O}$ as vertex and
with edges $\hat{a}(\tau)\supset a(\tau)$ and $\hat{b}(\tau)\supset b(\tau)$; it follows
 that $\mathcal{\hat{T}}(\tau)\supset \mathcal{T}(\tau)$.
We begin from a result requiring very weak regularity properties.

\begin{lemma}\label{boxu}
Assume that local uniqueness for the solutions of~\eqref{equazione} is ensured for any
trajectory starting from $\mathcal{T}(\tau) \setminus \{\boldsymbol{O} \}$.

Suppose that the flow on $a(\tau)\cup b(\tau)$ points towards the interior of $\mathcal{T}(\tau)$, and on
$o(\tau)$ points towards the exterior of $\mathcal{T}(\tau)$ for any $t \le \tau$. Assume further that the flow on
 $\{\boldsymbol{A}(\tau), \boldsymbol{B}(\tau) \}$ points towards the interior of $\mathcal{\hat{T}}(\tau)$ for any $t \le \tau$.
 Finally suppose that if a solution $\boldsymbol{x}(t)$ of~\eqref{equazione} satisfies $\boldsymbol{x}(t) \in T(\tau)$ for any $t \le \tau$,
 then $\lito \boldsymbol{x}(t)=\boldsymbol{O}$.

 Then there is a compact connected set $\bar{W}^u(\tau)\subset \mathcal{T}(\tau)$ such that $\boldsymbol{O} \in \bar{W}^u(\tau)$,
 $\bar{W}^u(\tau) \cap o(\tau)  \ne \emptyset$, with the following property:
$$
%\begin{equation}\label{characu}
\bar{W}^u(\tau) \subset  \{ \Q  \mid   \lito \boldsymbol{x}(t,\tau;\Q) = \boldsymbol{O} \, , \;
 \boldsymbol{x}(t,\tau;\Q) \in \mathcal{T}(\tau) \; \textrm{ for any $t \le \tau$} \} \,.
%\end{equation}
$$
\end{lemma}

\begin{proof}
This Lemma is proved in~\cite[§~3]{F6}, see also~\cite[Lemma~3.5]{Fduegen}: the reasoning relies on
a connection argument and a topological idea developed in~\cite[Lemma~4]{PZ}.
\end{proof}
Obviously the same idea can be applied to construct stable sets.

\begin{lemma}\label{boxs}
Assume that local uniqueness for the solutions of~\eqref{equazione} is ensured for any
trajectory starting from $\mathcal{T}(\tau) \setminus \{\boldsymbol{O} \}$.

Suppose that the flow on $a(\tau)\cup b(\tau)$ points towards the exterior of $\mathcal{T}(\tau)$, and on
$o(\tau)$ points towards the interior of $\mathcal{T}(\tau)$ for any $t \ge \tau$. Assume further that the flow on
 $\{\boldsymbol{A}(\tau), \boldsymbol{B}(\tau) \}$ points towards the exterior of $\mathcal{\hat{T}}(\tau)$ for any $t \ge \tau$.
 Finally suppose that if a solution $\boldsymbol{x}(t)$ of~\eqref{equazione} satisfies $\boldsymbol{x}(t) \in T(\tau)$ for any $t \ge \tau$,
 then $\lit \boldsymbol{x}(t)=\boldsymbol{O}$.

 Then there is a compact connected set $\bar{W}^s(\tau)\subset \mathcal{T}(\tau)$ such that $\boldsymbol{O} \in \bar{W}^s(\tau)$,
 $\bar{W}^s(\tau) \cap o(\tau)
 \ne \emptyset$, with the following property:
$$
%\begin{equation}\label{characs}
\bar{W}^s(\tau) \subset  \{ \Q  \mid   \lit \boldsymbol{x}(t,\tau;\Q) = \boldsymbol{O} \, , \;
 \boldsymbol{x}(t,\tau;\Q) \in \mathcal{T}(\tau) \; \textrm{ for any $t \ge \tau$} \} \,.
%\end{equation}
$$
\end{lemma}

If we are in the position to apply invariant manifold theory for non-autono\-mous systems, clearly we find that
these sets are manifolds. So we get the following.

\begin{lemma}\label{box}
Assume that we are in the hypotheses of Lemma~\ref{boxu}, respectively of Lemma~\ref{boxs}. Assume further that
  $F$ is $C^1$ and it is continuous in $x$ uniformly with respect to $t \in \RR$. Suppose that the linearized system  admits
exponential dichotomy, i.e. there are projections $\mathcal P^+$ and $\mathcal P^-$ of rank $1$ such that~\eqref{dichotu} and
\eqref{dichots} hold, so that $\boldsymbol{O}$ admits unstable and stable manifolds $W^u(\tau)$ and $W^s(\tau)$ for any $\tau \in \RR$. Then
the set
$\bar{W}^u(\tau)\subset (\mathcal{T}(\tau)\cap W^u(\tau))$ constructed in Lemma~\ref{boxu}, resp. the set $\bar{W}^s(\tau)\subset (\mathcal{T}(\tau)\cap W^s(\tau))$ constructed in Lemma~\ref{boxs}, is a connected $1$-dimensional manifold.
\end{lemma}

\end{document}